\documentclass[12pt,  reqno]{article}
\usepackage{amsmath,amsfonts,amssymb,amsthm}

\usepackage{amsmath}%
\usepackage{amsfonts}%
\usepackage{amssymb}%
\usepackage{graphicx}
\usepackage{bbm}

\newtheorem{theorem}{Theorem}[section]

\newtheorem{corollary}{Corollary}[section]

\newtheorem{lemma}{Lemma}[section]

\newtheorem{proposition}{Proposition}[section]

\numberwithin{equation}{section}
\newcommand*{\pmax}{p_\mathrm{max}}
\newcommand*{\betap}{\beta_{\bp}}
\newcommand*{\edense}{E_\mathrm{dense}}
\newcommand*{\efast}{E^n_\mathrm{fast}}
\newcommand{\ubx}{\mathbf{X}}
\newcommand{\bp}{\boldsymbol{p}}
\newcommand{\bq}{\boldsymbol{q}}
\newcommand{\br}{\boldsymbol{r}}
\newcommand{\qmax}{q_{\mathrm{max}}} 
 
\newcommand{\R}{\mathbb{R}}
\newcommand{\N}{\mathbb{N}}
\newcommand{\BP}{\mathbb{P}}
\newcommand{\BX}{\mathbb{X}}
\newcommand{\BZ}{\mathbb{Z}}
\newcommand{\BY}{\mathbb{Y}}
\newcommand{\BE}{\mathbb{E}}
\newcommand{\BT}{\mathbb{T}}
\newcommand{\tT}{\tilde{\mathbb{T}}}

\newcommand{\bS}{{\bf S}}
\newcommand{\bB}{{\bf B}}
\newcommand{\cS}{{\mathcal S}}
\newcommand{\eps}{\varepsilon}
\newcommand{\bea}{\begin{eqnarray}}
\newcommand{\eea}{\end{eqnarray}}
\newcommand{\bean}{\begin{eqnarray*}}
\newcommand{\eean}{\end{eqnarray*}}
\newcommand{\1}{{\bf 1}}

\begin{document}

\title{Percolation of even sites for enhanced random sequential adsorption}
\author{Christopher J. E. Daniels\thanks{Department of Mathematical Sciences,
 University of Bath, Bath BA2 7AY, United Kingdom.
Email: \texttt{ch.dnls@gmail.com}. Supported by an EPSRC studentship}
 ~and
Mathew D. Penrose\thanks{Department of Mathematical Sciences, University of
 Bath, Bath BA2 7AY, United Kingdom.
Email: \texttt{m.d.penrose@bath.ac.uk}
}}

\maketitle
\begin{abstract} 
Consider random sequential adsorption on a chequerboard lattice
with arrivals at rate $1$ on light squares and at rate $\lambda$
on dark squares. Ultimately, each square is either occupied, or blocked by
an occupied neighbour. Colour the occupied dark squares and blocked light
 sites {\em black}, and the remaining squares {\em white}.
Independently
at each meeting-point of four squares, allow diagonal
connections between black squares with probability $p$; 
otherwise allow diagonal connections between white squares.
 We show that there is a critical surface of pairs $(\lambda, p)$, containing 
the pair $(1,0.5)$, such that for $(\lambda, p)$ 
lying above (respectively, below) the critical surface the
 black (resp. white) phase percolates, 
and on the critical surface neither phase percolates.
\end{abstract}
\textbf{Key words:} Dependent percolation, random sequential adsorption, 
critical surface
\newline
\textbf{MSC:} 60K35, 82B43

\maketitle

\section{Introduction}
\label{secintro}
Random sequential adsorption (abbreviated RSA throughout this paper)
 is a term for a family of probability models for irreversible particle deposition. Particles arrive at random locations and times onto a surface, and if accepted a  particle blocks nearby locations on the surface from accepting future arrivals. 
Such models are of physical interest, as a modal for coating of a
surface; see for example \cite{Evans,Privman}.
We consider
 a discrete version of RSA on the
 initially empty
 integer lattice 
$\mathbb{Z}^2$, with the arrival time at a lattice site $x$ given by
 an exponential random variable $T_x$ with parameter $\lambda_x$, with 
$\left( T_x \right)_{x \in \mathbb{Z}^2}$ independent. 
All sites are either {\em empty}, {\em occupied} or {\em blocked}; 
an arrival at an empty site $x$ causes it to become
permanently occupied and all 
 adjacent sites (that is, sites $y$
 such that $|x-y|=1$ where $|\cdot|$ denotes the
 Euclidean norm) to become permanently blocked.
 If $\sup_x \lambda_x < \infty$ this model is well defined;
 see \cite{PenroseRSANotFail}. On this lattice we define the even 
(respectively, odd) sites to be those at an even
(respectively, odd)
 graph distance from the origin.

Ultimately, each site will be either occupied or blocked.
 The distribution of the occupied and blocked sites in this ultimate state is called the jamming distribution; under the jamming distribution the sites of $\mathbb{Z}^2$ are divided into an even phase and an odd phase,
 where the even phase consists of occupied even sites and blocked odd sites. Site percolation of the even phase was considered in \cite{RSAPaper}, in the case
 where for some $\lambda >0$ we have
$\lambda_x=1$ for odd $x$ and $\lambda_x=\lambda$ for even $x$.
The even phase is {\em monotone} in $\lambda$; that is, for 
$0 \leq \lambda < \lambda'$
there exists coupled realisations of the process just described with
parameter $\lambda$ and with parameter $\lambda'$, such that
the even phase for parameter $\lambda$ is 
contained in the 
 even phase for parameter $\lambda'$.

Penrose and Rosoman \cite{RSAPaper} proved that the critical parameter 
$\lambda$ for RSA on the integer lattice $\mathbb{Z}^2$ is strictly
 greater than $1$. The proof of this uses an {\em enhanced RSA}
 (denoted eRSA below)
 model on a new lattice called $\Lambda$ throughout this paper.
We associate with each site $x\in\mathbb{Z}^2$ a site $x^\prime := x+(1/2,1/2)$.
The lattice $\Lambda$ has
vertex set $\cup_{x \in \mathbb{Z}^2} \{x,x^\prime\}$,
with an edge between sites $x\in\mathbb{Z}^2$ and $y\in\mathbb{Z}^2$ if
$|x-y|=1$,
and an edge between 
$x^\prime$ and $y$ if $|x^\prime-y|=\frac{\sqrt{2}}{2}$ (here
$|\cdot|$ is the Euclidean distance). 
 We refer to the added sites $x^\prime$ as diamond sites, and the original
 sites $x$ as octagon sites;
each octagon site has degree 8 and each diamond site has degree 4
 (see Fig \ref{fig:RSA tiling}). 
\begin{figure}[htbp]
	\centering
		\includegraphics{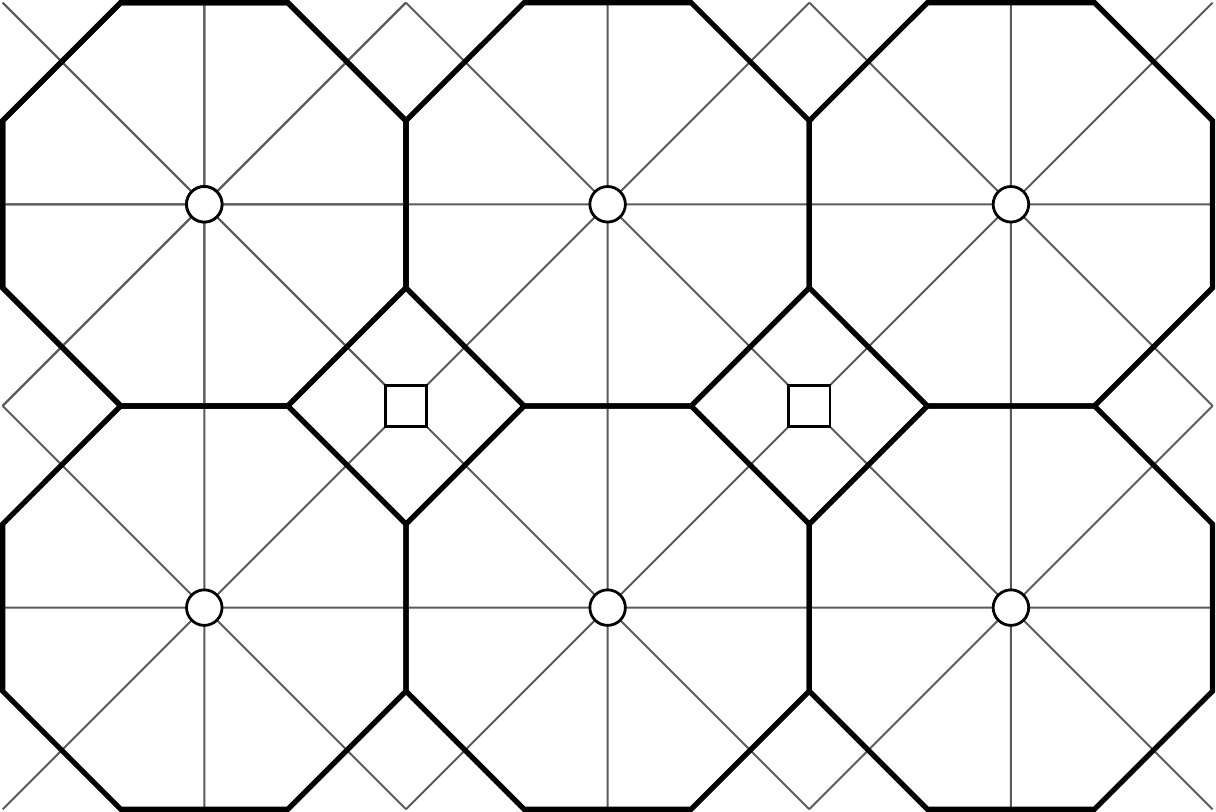}
		\caption{A section of the lattice used for enhanced RSA
and the associated tiling of $\R^2$.
The small circles represent octagon sites and the small squares are
diamond sites.
}
	\label{fig:RSA tiling}
\end{figure}

 We introduce  an {\em enhancement parameter} $p\in[0,1]$.
Each of the diamond sites $x^\prime$ is independently taken
 to be in the even phase with probability $p$, and otherwise in the odd phase.
Considering 
percolation on this new lattice, where a site is considered {\em black} if it is
 in the even phase and otherwise {\em white}, we say the even (resp. odd) phase
 {\em percolates} if there is an infinite component of 
black (resp. white) sites.

Taking the special case with $p=1$ amounts to always allowing diagonal
connections between black sites of $\BZ^2$, and never  allowing
diagonal connections between white sites. Taking $p=0$ 
amounts to the opposite.
 Considering the enhanced model enables us to interpolate continuously
between these extremes.
Moroever, this model enjoys a duality relation whereby
the even phase for parameters $(\lambda,p)$ has
the same distribution as the odd phase for parameters
$(1/\lambda,1-p)$ (see Lemma \ref{lemdual} below).

In this paper we consider the enhanced model in its own right, with 
a further parameter $\lambda \in \R_+ := (0,\infty)$ and with
 $\lambda_x=1$ for odd $x$ and $\lambda_x=\lambda$ for even $x$.
For $p \in[0,1]$ we define the critical values
\begin{align*}
\lambda^+_c(p)  
        & := \inf\{\lambda : \text{ the even phase percolates in eRSA}
(\lambda, p) \};
\\
\lambda^-_c(p)  &:= \sup\{\lambda : \text{ the odd phase percolates in eRSA}
(\lambda,p)  
\}.
\end{align*}
 It is natural to ask whether these values coincide,
and if so, to try
 to understand the behaviour of the critical
 surface in $(\lambda,p)$-space for this model; for example, the symmetry
 suggests that the pair $(\lambda=1,p=1/2)$ should be critical. 
Our main result provides some information on these issues.

\begin{theorem}
\label{thm:invCrit}
(i) For each $p \in [0,1]$ we have
$\lambda_c^-(p) \leq \lambda_c^+(p)$, with
equality whenever $ 0 < p < 1$.

(ii) For each $p \in [0,1]$
there is no percolation
of the even phase for
  eRSA with parameters
$(\lambda^+_c(p),p)$, and
 no percolation
of the odd phase for
  eRSA with parameters
$(\lambda^-_c(p),p)$.

(iii) It is the case that $\lambda_c^+(1/2)=1$.

(iv)
For any $\varepsilon \in (0,1/2)$, the functions
 $\lambda_c^+:[\varepsilon, 1] \to \mathbb{R}_+$
and
$\lambda_c^-:[0,1-\varepsilon] \to \mathbb{R}_+$
 are strictly decreasing and Lipschitz, and
the inverse of the function
 $\lambda_c^+:[\varepsilon, 1-\eps] \to [\lambda_c^+(1-\eps) , 
\lambda_c^+(\eps) ]$
is also strictly decreasing and Lipschitz.
\end{theorem}

We conjecture that $\lambda_c^+(p) = \lambda_c^-(p)$ 
for all $p \in [0,1]$ but we prove  this
only for $p \in (0,1)$. It is clear from the theorem
that the inverse function of $\lambda_c^+(\cdot)$ is the function
$p_c^+(\cdot)$ defined by
\begin{align*}
p_c^+(\lambda) 
  & := \inf\{p : \text{ the even phase percolates in eRSA}
(\lambda, p) \}.
\end{align*}

An outline of the proof will be provided in Section \ref{secProof}
with details filled in in subsequent sections.
Most of the work goes into showing that
if the {\em odd} phase does {\em not} percolate at a certain $(\lambda,p)$
(Assumption A), then after an arbitrarily small increase in either 
$\lambda $ or $p$ the {\em even} phase {\em does} percolate (Conclusion B). 
The strategy to prove this
goes as follows. 
Under Assumption A,
we shall adapt known methods to deduce 
that the even phase crosses an arbitrarily large rectangle
 of aspect ratio
3 the long way, with non-vanishing probability. Then using a suitable
{\em sharp thresholds} result for increasing events in
a finite product space (presented in Section \ref{secthresh},
and perhaps of independent interest),
we shall deduce in Proposition \ref{critequal}
 that after increasing $\lambda$ or $p$ we have  
a crossing of such a rectangle with probability close to 1, and then
a standard comparison with 1-dependent  percolation yields Conclusion B. 

To use our sharp thresholds result, we shall
  discretize time.
We shall demonstrate that the approximation error
involved in the discretization 
 can be compensated for with a slight increase in the parameter $\lambda$
or $p$ that
 vanishes as the size of the rectangle approaches infinity.
We do this using the method of {\em essential  enhancements}
(see for example \cite{AG}) to show that the
effect of the discretization parameter is comparable to that
of the enhancement parameter $p$. 

Our strategy outlined above 
 is    related to  to a method used by Bollob\'as and  Riordan
 in \cite{VoroPerc} to prove that the critical value for Voronoi
 percolation in the plane is $1/2$, but is distinguished 
by our  use of essential-enhancement techniques  in the last
step rather than the coupling construction appearing at a 
comparable stage in \cite{VoroPerc}. This method might be of 
use elsewhere. Indeed,   we believe that these methods are likely to  be relevant
to showing similar results on existence of a sharply defined
and smooth critical surface for 
other percolation models having two
or more parameters 
and long-range dependence, 
provided correlations are sufficiently rapidly decaying
(in the present instance this holds because of Lemma
\ref{affectlem}  below),
and identifying actual critical values for such models  with
sufficient symmetry.

For example, consider random sequential  deposition of monomers 
(at rate 1) and dimers (at rate $\alpha$) onto the vertices
of the triangular lattice with each monomer accepted
if it arrives at a previously unoccupied site,
and each  dimer accepted if it arrives at a
previously unoccupied pair of neighbouring sites. Suppose each
monomer (respectively dimer) is black with probability $p \in [0,1]$
(respectively $q \in [0,1]$). Ultimately all sites will
be occupied, and 
for fixed $\alpha$
we would expect that our methods could be adapted to show that there
 is a smooth critical
surface in $(p,q)$-space passing through $(1/2,1/2)$.

As another example, consider sequential deposition
of hard and soft particles, where
the hard particles exclude each other in RSA fashion, and any  point 
 not occupied by a hard particle acquires
the colour of the first soft particle to arrive
covering it; suppose hard particles
are black with probability $p$ and soft particles
are black with probability $q$. This could be considered, with
deposition either
on the vertices of the triangular lattice with hard particles
excluding each other from neighbouring sites,
or in the continuum $\R^2$  with the particles given by unit
disks (or some other shape).
Again, it may be possible to adapt our methods to these models.

The continuum version of the last model without the hard particles
(and therefore with finite range dependences)
amounts to the so-called `confetti percolation' or `dead leaves' model,
for which similar questions have been considered  
in \cite{Hirsch} and  \cite{Muller}.
In the latter paper, M\"uller deploys a different sharp threshold
type result with weaker symmetry requirements; it would
be very interesting to explore the possible application of
those ideas in models such as those mentioned above.

\section{Proof of Theorem \ref{thm:invCrit} }
\label{secProof}

In this section we prove our theorem, but with the proof of
 certain key steps deferred to later sections. We first assemble some
 facts based adapting known methods
to eRSA. 

For $x,y \in \BZ^2$,
we shall say that the site $x$ {\em affects}  the site $y$ if
 there is some self-avoiding
 path in $\mathbb{Z}^2$ starting at a neighbour of $x$ 
(note: not at $x$ itself)
 and ending at $y$,
 such that if the odd sites along this path are listed in order as
 $x_1, x_2,\ldots,x_m$, then $T_{x_1}\leq T_{x_2}\leq \cdots \leq T_{x_m}$.
\begin{lemma}
\label{affectlem}
Let $x,y \in \BZ^2$ with $x \neq y$.
With probability 1,
 if $x$ does not affect $y$, then no change to the arrival time
 at $x$ with all other arrival times remaining fixed can alter the state of $y$.
Moreover,  if
 $d(x,y)$ denotes the graph distance between $x$ and $y$ in $\BZ^2$, then
\bea
\BP [ x ~ {\rm affects } ~ y]
 \leq \frac{4^{d(x,y)}}{\lfloor d(x,y) /2 \rfloor   ! }.
\label{eqaffects}
\eea
\end{lemma}
\begin{proof}
First note that (\ref{eqaffects}) follows easily by the union bound.

Partition the sites of $\BZ^2 $ into {\em generations}
$G_0,G_1,G_2,\ldots$, defined as follows.
Set $G_0 = \{x\}$. Inductively, suppose for some $k$ that
$G_0,G_1,\ldots,G_k$ have been defined.
 For each $z \in \BZ^2 \setminus \cup_{i=0}^k G_i$ we put
$z \in G_{k+1}$ if and only if $T_z < T_{w}$ for all 
 $w \in \BZ^2 \setminus \cup_{i=1}^k G_i$ with $w$
 neighbouring $z$.
Using   (\ref{eqaffects}) and  the first Borel-Cantelli lemma, one
can show that
 with probability 1, the sets
$G_0,G_1,G_2,\ldots$ do indeed partition $\BZ^2$. 

We now prove the first assertion of the lemma, by induction on
the generation containing $y$. If $y \in G_1$ and $x$ does
not affect $y$, then $x$ is not a neighbour of $y$ so
all neighbours of $y$ have later arrival times than $T_y$,
and therefore $y$ becomes occupied  rather than blocked regardless
of the arrival time at $x$.

For the inductive step, suppose for some $k \in \N$ that
 the assertion of the lemma holds for all $y\in \cup_{i=1}^k G_i$.
Suppose that $z \in G_{k+1}$, and that $x$ does not 
affect $z$. Enumerate the neighbouring sites
of $z$ with earlier arrival times than  $z$ as
$y_1,\ldots,y_j$. These sites must all lie in $\cup_{i=0}^k G_i$
and moreover are not affected by $x$ (else $z$ would also be affected
by $x$). Hence by the inductive hypothesis, the occupied/blocked
status of sites $y_1,\ldots,y_j$ is not affected by any change to
the arrival time at $x$, and hence the status of site $z$ is
also not affected by such a change ($z$ is occupied if all
of $y_1,\ldots,y_j$ are blocked). This completes the induction. 
\end{proof}

Let $\BP_{\lambda,p}$ denote the probability measure  
associated with the enhanced RSA model
 where $\lambda$ is the rate of arrivals at even
sites and $p$ is the enhancement parameter.
We next provide the {\em Harris-FKG} inequality for this
model.  For any two black/white colourings $\alpha,\beta$ of 
the  vertices of $\Lambda$ (i.e., of the faces of the tiling),
let us write $\alpha \prec \beta$ if the set of black
sites in $\alpha $ is contained in the set of black
sites in $\beta$.
Let us say that an event $E$, defined in terms of
the  colouring  induced by the RSA model, 
is {\em black-increasing} if it has
the following property:
 for any two colourings $\alpha, \beta $ with
 $\alpha \prec \beta$,
if $\alpha \in E$ then $\beta \in E$.
For example, $H_{n,\rho}$ is black-increasing.
 Similarly, we say $E$ 
is {\em white-increasing} if for 
  any two colourings $\alpha, \beta $ with $\beta \prec \alpha$,
if $\alpha \in E$ then $\beta \in E$. 
\begin{lemma}
\label{lemHarris} {\rm [Harris-FKG inequality.]}
Let $\lambda >0,p \in [0,1]$. 
If $E$ and $F$ are both black-increasing events or
are  both white-increasing events then
$\BP_{\lambda,p}(E \cap F) \geq 
\BP_{\lambda,p}(E )  
\BP_{\lambda,p}( F)$.  
\end{lemma}
\begin{proof}
The Harris-FKG inequality  
for RSA is given in section 5 of \cite{PenSud} and
may then be deduced
 in the  enhanced RSA model using the independence of the
 enhancement variables.
\end{proof}
We shall refer to the following lemma as a {\em duality} relation.
\begin{lemma}
\label{lemdual}
Let $\lambda >0, p \in [0,1]$. Then the even phase
of eRSA with parameters $(\lambda,p)$ percolates,
if and only if the odd phase of 
of eRSA with parameters $(1/\lambda,1-p)$ percolates.
\end{lemma}
\begin{proof}
Consider first the eRSA process with parameters $(\lambda,p)$.
Now re-scale time by multiplying all arrival times by a factor
of $\lambda$; the rescaled arrival times are exponential
with rate $1$ at even sites and rate $1/\lambda$ at odd sites.
If we then also interchange the colours, then the new set of
black sites is
a realization of eRSA with parameters $(1/\lambda,1-p)$.  
\end{proof}

Given a rectangle $R=[a,b]\times[c,d]$, and $r  > 0 $,
 define $\edense(R,r)$ to be the event that no
 site in $R$ is affected by any site outside $[a-r,b+r]\times[c-r,d+r]$.
Using (\ref{eqaffects}), one can readily prove the following which
has appeared previously as Lemma 3.3 of \cite{RSAPaper}.

\begin{lemma}
\label{3.3}
 Let $\lambda>0$, $\rho\geq1$. Given $s>0$, let 
$R_s=[1,\lfloor s\rfloor]\times[1,\lfloor\rho s\rfloor].$ Then 
$\BP_\lambda[\edense(R_s,2\lfloor s^{1/2}\rfloor)]\to1$ as 
$s\to\infty$. Moreover, $\edense(R,r)$ depends only on the arrival
 times within the larger rectangle. 
\end{lemma}

We now discuss certain box crossings. We
 construct a dependent face
 percolation model on a truncated square tiling (shown by the darker lines
in figure \ref{fig:RSA tiling})
as follows: colour the
 octagon centred at $x\in\mathbb{Z}^2$ black if $x$ is in the even phase,
 otherwise colouring it white. 
As in Section 1 we denote the diamond at the top right corner of the octagon
 centred at $x$ by $x^\prime$, and colour it black 
if the $x^\prime$ is in the even phase, and
 white otherwise.
Given  $\rho \in (0,\infty)$ and
$n \in \N$ with $\rho n \geq 1$, let
 $H_{n, \rho}$ denote the event that there is a horizontal black crossing
 of the rectangle
 \[R(2n,\rho):=[-\lfloor \rho n \rfloor, \lfloor \rho n \rfloor-1]\times
[-n, n-1]
\] 
and set 
\[h_\rho(n,\lambda,p) := \BP_{\lambda,p}(H_{n, \rho}).\]
Also, define $h^\prime_\rho(n,\lambda,p)$ similarly but
in terms of a white crossing. That is,
 $h^\prime_\rho(n, \lambda, p)$ denotes the probability
 that there is a horizontal white crossing of an arbitrary fixed 
$2\lfloor \rho n \rfloor$ by $2 n$ rectangle for
 eRSA with parameters $\lambda$ and $p$.  Note that
any given rectangle possesses either a horizontal black crossing or 
vertical white crossing.

\begin{lemma}
\label{LemFiniteBox}
There exist  constants $\kappa >0$, and 
$n_0 \in \N$, such that the even phase percolates
if there exists $n \geq n_0$ with
 $h_3 (n, \lambda, p)>1-\kappa$. 
\end{lemma}
\begin{proof}
This can be proved by a similar method to Theorem 1.1 of \cite{VoroPerc},
namely comparison with 1-dependent bond percolation along with
use of Lemma \ref{3.3}.
\end{proof}

The next ingredient is an RSW type result relating the
probability of crossing a large box of one (fixed) aspect
ratio, to the probability of crossing a box of a different
aspect ratio.

\begin{lemma}
\label{LimResBetter} 
Let $\lambda>0$, $p\in [0,1]$, $\rho>0$ be fixed.
 If $\limsup_{n\to\infty}h_\rho(n,\lambda,p)>0$ then
 $\limsup_{n\to\infty}h_{\rho^\prime}(n,\lambda,p)>0$ for all $\rho^\prime>0$.
 If $\limsup_{n\to\infty}h^\prime_\rho(n,\lambda,p)>0$ then
\linebreak
 $\limsup_{n\to\infty}h_{\rho^\prime}^\prime(n,\lambda,p)>0$ for all $\rho^\prime>0$.
\end{lemma}
\begin{proof}
A weaker version of this result (with $\liminf$ rather than $\limsup$
in the hypothesis, and with $\rho =1$) is given by
the proof of Proposition 3.2 in \cite{RSAPaper}, based on 
 that of Theorem 4.1 of \cite{VoroPerc}.
The details on how to convert the proof to
 the stronger statement given here can
 be found in \cite[Section 4]{BBV}.
The argument uses the rapid decay of correlations (which follows 
 from Lemma \ref{affectlem}), the Harris-FKG inequality,
and the invariance of the model under 90 degree rotations
and under reflections in the $y$-axis.
\end{proof}

Note that by symmetry, $h_1(n,1,1/2)=1/2$, and therefore we have for all
 $\rho'>0$ that $\limsup_{n\to\infty} h_{\rho'} (n,1,1/2)>0$.
Using the method of proof of the recent result in \cite{Tassion}, it should be
possible to show that in fact $\liminf_{n \to \infty} h_{\rho'} (n,1,1/2)>0$,
but we do not need this. On the other hand, in applying Lemma
\ref{LimResBetter} we will need the case with $\rho=1/3$ as
well as the case with $\rho=1$. 
\begin{lemma}
\label{percCondition}
For eRSA with parameters $(\lambda,p)$, the even phase
 percolates if and only if $\lim_{n\to\infty}h_3(n,\lambda,p)=1$;
the odd phase percolates
 if and only if $\lim_{n\to\infty}h_{1/3}(n,\lambda,p)=0$.
\end{lemma}
\begin{proof}
By Lemma \ref{LemFiniteBox}, it is immediate that
 $\lim_{n\to\infty}h_3(n,\lambda,p)=1$ implies percolation of the even phase.

Now suppose that $\liminf_{n\to\infty}h_3(n,\lambda,p)<1$; then
 $\limsup_{n\to\infty}h_{1/3}^\prime(n,\lambda,p)>0$. By
 Lemma \ref{LimResBetter}, hence
$\limsup_{n\to\infty}h_{3}^\prime(n,\lambda,p)>0.$ 
We can thus find a
 sequence $(n_i)_{i\in\mathbb{N}}$ with $n_{i+1}>4n_{i}$ such that
 $\liminf h_3^\prime(n_i, \lambda, p)>0.$
For $i \in \N$ define rectangles $R_{i,j}, 1 \leq j \leq 4$, by
 \begin{align*}
R_{i,1} &= [-3 n_i, 3 n_i] \times [-3 n_i, -n_i]\\
R_{i,2} &= [n_i, 3 n_i] \times [-3 n_i, 3 n_i]\\
R_{i,3} &= [-3 n_i, 3 n_i] \times [n_i, 3 n_i]\\
R_{i,4} &= [-3 n_i, n_i] \times [-3 n_i, - 3n_i].
\end{align*}
Let $E_i$ denote the event that all four rectangles $R_{i,j}$ contain a
 long way crossing in the odd phase and that 
 $ \cap_{j=1}^4 \edense(R_{i,j},n_i/8)$ holds. 
 By Lemma \ref{3.3}, 
$\BP_{\lambda, p}\left(E_{\text{dense}}(R_{i,j},n_i/8)\right)\to 1$
as $i \to \infty$,
 and hence by the Harris-FKG inequality,
 $\liminf \BP_{\lambda, p}(E_i) > 0$.
 Since the events $(E_i)_{i\in\mathbb{N}}$ are independent,
it follows that almost surely at least one of them occurs, and hence
the cluster containing the origin in the even phase is almost surely finite.

The last part is proved similarly.
%
\end{proof}

The next two propositions are key ingredients in the proof of Theorem 
\ref{thm:invCrit}; we defer their proof to Sections \ref{secKeys}
and \ref{seckey2}. The first of these  says that the effect
of a small change in $\lambda$ on box crossing probabilities
 is comparable to that of
a small change in $p$.

\begin{proposition}
\label{cor:halfCorrection} 
Let $\varepsilon \in (0,1/2)$.
 Then there is a constant $c_1=c_1(\varepsilon) \in (0,\infty)$
 such that for any
$(n,\lambda,p) \in \N \times [\varepsilon,1/ \varepsilon] \times
 [\eps,1-\eps] $
 we have
\begin{equation}
 c_1^{-1} \frac{\partial h_3(n,\lambda,p)}{\partial \lambda} \leq
 \frac{\partial h_3(n,\lambda,p)}{\partial p} \leq c_1
 \frac{\partial h_3(n,\lambda,p)}{\partial \lambda},
\label{eq:LambdaComp}
\end{equation}
and moreover the second inequality of 
(\ref{eq:LambdaComp}) holds for any
$(n,\lambda,p) \in \N \times [\varepsilon,1/ \varepsilon] \times
 [0,1] $.
\end{proposition}

The last key ingredient says that if (at some $(\lambda,p)$) we have
non-vanishing probability of crossing a large rectangle
of fixed aspect ratio, then after a slight increase of
either $\lambda $ or $p$ we have probability close to 1 of
crossing a rectangle of aspect ratio 3 the long way. 

\begin{proposition}
\label{critequal} 
Let $\lambda >0$, $ p \in (0,1)$ and $\eps >0$ 
with $p + \eps <1$.
  Suppose for some $\rho >0$ that
 $\limsup_{n \to \infty} h_{\rho} (n,\lambda,p) >0$.   
 Then 
 \begin{equation}
\limsup_{n\to\infty} h_3(n,\lambda +\eps,p)
=1
\label{0815b}
\end{equation}
and 
 \begin{equation}
\limsup_{n\to\infty} h_3(n,\lambda,p+\varepsilon)
=1.
\label{0815a}
\end{equation}
\end{proposition}

We can now prove of Theorem
 \ref{thm:invCrit}, using the strategy outlined in Section \ref{secintro}.

\begin{proof}[Proof of Theorem \ref{thm:invCrit}]
Let $n_0 \in \N $ and $\kappa >0$ be as in Lemma \ref{LemFiniteBox}.
Let $S$ denote the set of $(\lambda,p)$ such that
$h_3(n,\lambda,p) > 1- \kappa$ for some $n \geq n_0$.
Since $h_3(n,\lambda,p)$ is continuous in $\lambda$ and $p$ for any fixed $n$,
the set $S$ is open in $(0,\infty) \times [0,1]$, 
and the even phase percolates for any $(\lambda,p) \in S$.
 Also, if $(\lambda, p) \notin S$ then
 $\limsup(h_3 (n, \lambda, p)) \leq 1-\kappa <1$, and thus there is no
 percolation by Lemma \ref{percCondition}.
Hence, $S$ is the set of $(\lambda,p)$ for which the
even phase percolates.

Similarly, with $S^\prime$ denoting the set of values
 $(\lambda, p)$ for which the odd phase percolates,
the set $S^\prime$ is also open in $(0,\infty) \times [0,1]$.
Also $S \cap S^\prime = \emptyset $ by Lemma
 \ref{percCondition}. Since 
 $\lambda^+_c(p) = \inf \{ \lambda : (\lambda,p) \in S\}$,
 and $\lambda^-_c(p) = \sup \{ \lambda : (\lambda,p) \in S'\}$,
this gives us the inequality
$\lambda_c^-(p) \leq \lambda_c^+(p) $.

For $(\lambda,p), (\lambda',p') \in (0,\infty) \times [0,1]$,
let us write 
 $(\lambda',p') \succ (\lambda,p)$ 
to mean that
$\lambda \leq \lambda'$ and $p \leq p'$ with at
least one of these inequalities being
strict.

Suppose
$(\lambda,p) \notin S'$ and $ 0 < p < 1$. Then
$
\limsup_{n\to\infty}
h_{1/3}(n,\lambda,p)>0
$
by Lemma \ref{percCondition}.
Hence by  Proposition \ref{critequal},
 for any $(\lambda',p') \succ (\lambda,p) $
 we have 
$\limsup_{n\to\infty} h_3(n,\lambda', p')=1$. Therefore
$(\lambda',p') \in S$. 
In other words, for $0 < p <1$ we have
\begin{equation}
(\lambda,p) \notin S' \Longrightarrow (\lambda',p') \in S
\quad \forall
(\lambda',p') \succ (\lambda,p). 
\label{0815c}
\end{equation}
%
%
%
Hence
 for $0 < p <1$  we have  $\lambda_c^-(p) = \lambda_c^+(p) $.
Thus we have part (i) of our theorem, and part
(ii)
 follows from the fact that 
the sets $S$ and $S'$ are open.

Since $h_1(n,1,1/2) =1/2$ for all $n$, we have
that 
$\limsup h_3(n,1,1/2) <1$, so 
 by Lemma \ref{percCondition}
we have $(1,1/2) \notin S$.
Hence  by duality, also
$(1,1/2) \notin S'$,
and part (iii) follows.

For part (iv), the strict monotonicity
of $\lambda_c(\cdot)$
follows from (\ref{0815c}) and
the fact that $S$ is open. 
We next prove the Lipschitz continuity of $\lambda_c(\cdot)$.

Let $\eps \in (0,1/2)$.
By Theorem 2.1 of \cite{RSAPaper},
$\lambda_c^+(0) < 10$, and hence 
by duality, $\lambda_c^-(1)>0.1$. 
By Proposition \ref{cor:halfCorrection}, 
we can find
$c_1 \in (1,\infty)$ such that for
 $(\lambda,p,n) \in  [0.1,11] \times [0,1] \times \N$,
 we have the second inequality of (\ref{eq:LambdaComp}).

 Let $p \in [\eps,1]$.  Then
 for any $\lambda \in ( \lambda^+_c(p),10)$, we have
$\lambda \in S$ so we can find
$n \geq n_0$ such that $h_3(n, \lambda,p) \geq 1-\kappa$.
Then by
the second inequality of
 (\ref{eq:LambdaComp}),  for 
 such $n$
and for
$0 < \delta < \eps/c_1$
 we have
$$
h_3(n,\lambda + c_1 \delta, p- \delta) \geq
h_3(n,\lambda,p) \geq 1- \kappa
$$  
so that $(\lambda + c_1 \delta,p-\delta) \in S$ and hence
$\lambda_c^+(p-\delta) \leq \lambda_c^+(p) + c_1 \delta $.
This gives the Lipschitz continuity of $\lambda_c^+(\cdot)$
on $[\eps,1]$.

Now suppose $0 \leq p \leq 1-\eps$.
By duality (Lemma \ref{lemdual}) we have
 $ \lambda_c^-(p) = 1/ \lambda_c^+(1-p).  $
Thus for  $0 < \delta < \min (\eps, (100 c_1)^{-1})$,
using that $\lambda_c^+(1-p) < 10$
 we also have 
$$
\lambda_c^+(1-p - \delta) \leq \lambda_c^+(1-p) + c_1 \delta
\leq \frac{ 1}{\lambda_c^-(p) - c' \delta},
$$
for  $c'= 100 c_1$.
  Hence by duality again, 
$ \lambda_c^-(p+\delta ) \geq  \lambda_c^-(p) -  c' \delta.  $ 
This shows the Lipschitz continuity of $\lambda_c^-(\cdot)$
on $[0,1-\eps]$.

For $\lambda_c^+(1) < \lambda < \lambda_c^-(0)$, 
set $p_c^+(\lambda) := \inf \{p: (\lambda,p) \in S \}$.
By (\ref{0815c}) and the fact that $S$ is open,
the function $p_c^+(\cdot)$ is strictly decreasing.
By a similar argument to the above
(now using the first inequality of (\ref{eq:LambdaComp})), 
we may show the 
 Lipschitz continuity of $p_c^+(\lambda)$ as a 
function of $\lambda$ for 
 $\lambda_c^+(1) +\eps \leq \lambda \leq \lambda_c^-(0) - \eps$.
Thus the restriction of the function $\lambda_c^+(\cdot)$ 
 to the domain  $[1-\eps,\eps]$
has a Lipschitz inverse,
namely $p^+_c(\cdot)$.
%
%
\end{proof}

So far, we have used only the generic properties mentioned
 in the proof of Lemma \ref{LimResBetter}, along with
duality.  For the 
enhancement estimates 
required below to prove Propositions \ref{cor:halfCorrection} 
and \ref{critequal},
we shall require arguments more specific
to this particular model.

\section{A sharp thresholds result}
\label{secthresh}

The {\em sharp threshold property} 
\cite{FK} 
for increasing events
in $\{0,1\}^n$ 
says that for any such event and any fixed $\eta \in (0,1/2)$,
when $n$ is large the threshold 
value of $p$ above which
 the probability of such an event
 (under product measure with parameter $p$) 
exceeds $1- \eta$, is only slightly larger
than the corresponding threshold 
for the event to have probability 
at least $\eta$.

In Proposition \ref{SharpNM}
below, we present a similar threshold result
for events in $\{0,1,\ldots,k\}^n$ for any fixed $k$,
satisfying a symmetry assumption.
Such a result was given in 
 \cite{JohnMehlErrata} for the case $k=2$; we adapt
this to general $k$ and give a more detailed proof than that of
\cite{JohnMehlErrata}. Later, we shall use 
 Proposition \ref{SharpNM} to prove Proposition \ref{critequal}.

Let $k \in \mathbb{N}$. For
 $n, m\in\mathbb{N}$, a 
subset $E \subset \{0,1,...,k\}^n$ is said to have
 symmetry of order $m$ if there is a group action
 on $[n] :=\{1,2,...,n\}$ in which each orbit has size at
 least $m$, such that the induced action on $\{0,1,...,k\}^n$ preserves
 $E$; for instance, if $n$ is even then 
 a subset $E\subset\{0,1,...,k\}^{n^2}$ 
 which is preserved by even translations of the $n $ by $ n$
 torus $[n] \times [n]$ (identified with $[n^2]$)
 would have symmetry of 
order $n^2/2$.

Given 
a probability vector 
 $\bp =(p_0,p_1,...,p_k)$  
(i.e., a finite vector with nonnegative entries summing to 1),
we write $\BP_{\bp}$ for the probability measure
on $\{0,\ldots,k\}$ with probability mass function
$\bp$, and for $n \in \N$ 
we write
 $\BP_{\bp}^n$ for the $n$-fold product  of this probability measure
(a probability measure on $\{0,\ldots,k\}^n$).
We say that $E \subset \{0,1,\ldots,k\}^n$
is {\em increasing}, if for every $x= (x_1,\ldots,x_k)$ and 
$y=(y_1,\ldots,y_k)$ in  $\{0,1,\ldots,k\}^n$ such
that $x \in E$ and $y_i \geq x_i$ for $i\in \{0,1,\ldots,n\}$,
we have $y \in  E$.

Given probability vectors $\bp = (p_0,\ldots,p_k)$ and
$\bq = (q_0,\ldots,q_{k})$, we say that $\bq$ {\em dominates} $\bp$ if
for $j=0,1,2,\ldots,k-1$ we have $\sum_{i=0}^j (p_i -q_i) \geq 0$.
Note that $\bq$ dominates $\bp$, if and only if
 there are coupled random variables
$X,Y$ taking values in $\{0,1,\ldots,k\}$ such
that $X$ has distribution $\mathbb{P}_{\bp}$ and
 $Y$ has distribution 
 $\mathbb{P}_{\bq}$ and $Y \geq X$ almost surely.
Thus, if $\bq$ dominates $\bp$ then
 $\mathbb{P}_{\bq}^n(E) \geq 
 \mathbb{P}_{\bp}^n(E) $ for any $n \in \N$ and any increasing
$E \subset \{0,1,\ldots,k\}^n$.

\begin{proposition}\label{SharpNM}
Let $k,n, m \in\mathbb{N}$,
 let $\eta \in (0,1/2)$ and let $\gamma >0$.
Suppose $\bp = (p_0,p_1,\ldots,p_k)$ and $\bq = (q_0,q_1,\ldots,q_k)$ are
probability vectors such that $p_0 \geq \gamma$,
$p_k \leq 1-\gamma$ and $\bq$  
dominates $\bp + (-\gamma,0,\ldots,0,\gamma)$.
Let $\qmax$ denote the second largest of
the numbers $p_0,\ldots,p_{k-1},p_k+ \gamma$, and 
suppose also that
\begin{equation}
\label{SharpBound}
\gamma \log m
 \geq
 200k^2 
\log (1/\eta)
\qmax
\log(4/\qmax).
\end{equation}
Then for any increasing 
$E \subset \{0,1,\ldots,k\}^n$ with symmetry of order $m$,
and with $\mathbb{P}_{\bp}^n(E)>\eta$, we have
$\mathbb{P}_{\bq}^n(E)>1-\eta$.
\end{proposition}

The remainder of this section is devoted to proving this, via a series 
of lemmas.

Given a probability vector
 $\bp =(p_0,p_1,...,p_k)$, 
define
\bean
\label{eq:BetaDef} 
\betap(x) & := & \max \left\{ j \in \{0,\ldots,k\}: \sum_{i=0}^{j-1} p_i \leq x 
\right\}, ~~~~~ x \in [0,1). 
\eean
where by definition we set $\sum_{i=0}^{-1} p_i = 0$.
Define $\pmax(\bp)$ to be the {\em second largest} of
the numbers $p_0,p_1,\ldots,p_k$.
Given $\ell \in \N$,
let $h_\ell:[0,1) \to [0,1)$ be the function which inverts the $\ell$th digit
 of the binary expansion of a number (using the
 terminating expansion wherever there is a choice). Now we let
 $U$ be a uniform $(0,1)$ distributed random variable, and
for $f:\{0,1,\ldots,k \} \to \{0,1\}$ define  
\begin{align*}
w_{\ell,\bp}(f) &=
 \mathbb{P}\left[f \circ \betap(U) \neq f \circ \betap
 (h_\ell(U))\right], ~~~~ \ell \in \N;
\\
w_{\bp}(f)&=\sum_{\ell=1}^\infty 
w_{\ell,\bp}(f).
\end{align*}


\begin{lemma}\label{KeyBoundUnlim} 
Let $k \in \mathbb{N}$. 
Then for any probability vector $\bp=(p_0,p_1,\ldots,p_k)$
 with $p_i>0$ for all $i$, and any
 function $f:\{0,1,\ldots,k\}\rightarrow\{0,1\}$, we have
\begin{equation}
w_{\bp}(f)\leq 3 k^2 \pmax(\bp) \log(4/\pmax(\bp)).
\label{eqomegabound} 
\end{equation}
\end{lemma}
\begin{proof}

For $0 \leq j \leq k-1$  define
 $q_j :=\sum_{i = 0}^j p_i$ and
 $q_j^* :=\min(q_j,1-q_j)$.
Then for $\ell \in \N$,
$$
w_{\ell,\bp}(f) \leq 
2 \sum_{j=0}^{k-1} \mathbb{P}(U<q_j <h_\ell(U))
\leq 2 \sum_{j=0}^{k-1} \min(q_j^*,2^{-\ell}).
$$
Hence,
\begin{align*}
w_{\bp}(f) 
%
&\leq \sum_{j=0}^{k-1}
\left( 
 \sum_{\ell=1}^{\left\lfloor\log_2(1/q_j^*)\right\rfloor}2 q_j^* +
 \sum_{\ell=\left\lceil\log_2(1/q_j^*)\right\rceil}^\infty 2^{1-\ell} 
\right)
\\
&\leq \sum_{j=0}^{k-1} \left(2 q_j^* \log_2(1/q_j^* ) + 4 q_j^* \right) 
= \sum_{j=0}^{k-1} 2 q_j^* \log_2(4/q_j^* ).  
\end{align*}
%
%
By routine calculus  
$  p \log_2 (4/ p)$ is increasing in $p $ for $p \in (0,1)$.
Hence for all $j$ in the sum,
\bea
q_j^* \log_2(4/q_j^*) \leq
 q_j \log_2 (4/q_j) \leq \sum_{\ell=0}^j p_{\ell} \log_2 (4/p_\ell),
\label{wIneq1dLow}
\eea
and also
\bea
q_j^* \log_2(4/q_j^*) \leq
 (1 - q_j) \log_2 (4/(1 - q_j)) \leq \sum_{\ell=j+1}^k p_{\ell}
 \log_2 (4/p_\ell).
\label{wIneq1dHigh}
\eea
Choose $s \in \{0,1,\ldots,k\}$ such that 
$p_s = \max(p_0, p_1,\ldots,p_k)$.
Using \eqref{wIneq1dLow} for $j < s$
and  \eqref{wIneq1dHigh} for $j \geq s$, we obtain
\begin{eqnarray*}
w_{\bp}(f) & \leq &  
 k (k+1) \pmax (\bp)  \log_2 (4/\pmax(\bp)) . 
\end{eqnarray*}
Since $k(k+1) \leq 2 k^2$ and 
$ \log 2 > 2/3$,
the result (\ref{eqomegabound}) follows.
\end{proof}

Given $k,n \in \N$, given
$f:\{0,1,\ldots,k\}^n\rightarrow\{0,1\}$ and
$j \in \{0,\ldots,k\}$, and given $x =(x_1,\ldots,x_n)  \in 
\{0,1,\ldots,k\}^n, $ we say the $j$th coordinate of $x$
is {\em pivotal} for $f$ if there exists $y
=(y_1,\ldots,y_n)  \in \{0,1,\ldots,k\}^n, $ with $y_i = x_i $ for all $i \neq j$,
such that $f(x) \neq f(y)$. 
Given also a probability vector $\bp=(p_0,\ldots,p_k)$ we
define the {\em influence} $I_{f,\bp}(j)$ of the
$j$th coordinate on $f$ as the probability
that the $j$th coordinate of $X$ is pivotal for $f$,
where here $X$ is a random element of 
$ \{0,1,\ldots,k\}^n $ with distribution $\BP^n_{\bp}$.

\begin{lemma}
\label{leminfl}
Let $k,n \in \mathbb{N}$. 
For any probability vector $\bp=(p_0,p_1,\ldots,p_k)$
 with all $p_i>0$,
any function 
$f:\{0,1,\ldots,k\}^n\rightarrow\{0,1\}$, any
$q \in [ \pmax(\bp) ,1]$ and any $a \in (0,1/16]$, if
\begin{equation}
I_{f,\bp}(j)\leq a
 q^2(\log(4/q))^2,
~~~~~ \forall
j \in [n], 
\label{0721a}
\end{equation} 
then 
setting $t=\mathbb{P}^n_{\bp}(f^{-1}(1))
= \BE  f(X) $,
we have that 
\begin{equation}
\sum_{j=1}^n I_{f,\bp}(j)\geq 
 \frac{t(1-t)\log(1/a) }{24 k^2 q \log(4/q)}.
\label{0721b}
\end{equation}
\end{lemma}
Given Lemma \ref{KeyBoundUnlim}, 
the proof of Lemma \ref{leminfl} is similar to that of 
Lemma 2 of \cite{JohnMehlErrata}. However, the argument given
there (even in the ArXiv version, which has more detail than
the published version) is quite sketchy, and 
`not intended to be read on its own';
it relies on arguments from Theorems 3.1 and 3.4 of \cite{FK},
 and both of these papers 
 rely heavily on arguments from \cite{BKKKL},
which is itself rather concise. 
Moreover, none of these
papers is entirely free of minor errors, which
does not aid readability. Therefore 
to make this presentation more self-contained,
and also to give explicit constants in the bounds,
we think it worthwhile to give a detailed proof.
However, we defer it to the Appendix.

We now give the proof of Proposition \ref{SharpNM}, which is
adapted from that of Lemma 1 in \cite{JohnMehlErrata}.

\begin{proof}[Proof of Proposition \ref{SharpNM}]
  Note that $ \gamma \leq \min(p_0,p_k+\gamma) \leq \qmax$.
Therefore
 by the assumption (\ref{SharpBound}), $\log m \geq 200 (\log 2)
 \log ( 4/ \qmax)$. 
Hence $m \geq \qmax^{-9}$, and also  $m \geq 16^4$.

For $0 \leq h \leq \gamma$ set $\br(h) = \bp + (-h,0,\ldots,0,h)$.
Let $g(h) = \mathbb{P}_{\br(h)}^n(E)$.
By assumption, $\bq$ dominates $\br(\gamma)$. 
Therefore 
 $\mathbb{P}^n_{\bq}(E) \geq
 \mathbb{P}^n_{\br(\gamma)}(E) = g(\gamma)$.
We shall use a form of the Margulis-Russo formula, namely
\bea
g'(h) =  I_{f,\br(h)} : = \sum_{j=1}^n I_{f,\br(h)} (j), 
\quad \quad \forall h \in (0,\gamma),
\label{0806b}
\eea
where $f$ is the indicator of event $E$ and
 $I_{f,\bp}(j)$ is the influence of the
$j$th coordinate on the function $f$, as in Lemma \ref{leminfl}.
To see (\ref{0806b}),
for $h_1,\ldots,h_n \in (0,\gamma)$ let $u(h_1,\ldots,h_n)$ denote
the probability of event $E$ under the 
 measure $\prod_{i=1}^n \mathbb{P}_{\br(h_i)}$
 and for probability vectors $\bp_1,\ldots,\bp_n$ on $\{0,1\ldots,k\}$
and $j \in [n]$ let 
$I_{f,(\bp_1,\ldots,\bp_n)} (j)$ denote the probability that the $j$th 
coordinate of $X$ is pivotal for $f$, where
$X$ is a random element of $\{0,1\ldots,k\}^n$ with
distribution $\prod_{i=1}^n \bp_i$.
Then for
$j \in [n]$ and $\eps >0$ with $h_j + \eps < \gamma$, we
can find coupled $\{0,1,\ldots,k\}^n$-valued
 random vectors $X $ and $X'$ 
with respective distributions
 $\prod_{i=1}^n \mathbb{P}_{\br(h_i)}$
and 
 $\prod_{i=1}^n \mathbb{P}_{\br'_i}$, where
we set $\br'_i= \br(h_i)$ except for $i=j$, and
$\br'_j= \br(h_j + \eps)$, and such that
$\mathbb{P} [X= X'] =1-\eps $ and
if $X \neq X'$ then $X_j = 0$ and $X'_j=k$,
with $X_i = X'_i$ for all $i \neq j$.  
Then $f(X') \leq f(X)$, with
equality except when (i) $X \neq X'$ and
(ii) the $j$th coordinate of $X$ is pivotal for $f$.
Therefore
$$
\mathbb{P} [f(X') \neq f(X)]
 = \eps I_{f,(\br(h_1),\ldots,\br(h_k))} (j)
$$
so that
$
\frac{\partial}{\partial h_j} u(h_1,\ldots,h_n)
 =  I_{f,(\br(h_1),\ldots,\br(h_k))}, 
$
and then we obtain (\ref{0806b}) by the chain rule, since
 $g(h) = u(h,h,\ldots,h) $.


Next we show that  for $0 \leq h \leq \gamma$ we have
\bea
I_{f,\br(h)} \geq \frac{g(h)(1-g(h))\log m}{96 k^2 \qmax \log (4/\qmax) }. 
\label{0806c}
\eea
First suppose $I_{f,\br(h)}(j) \geq m^{-1/2}$ for some $j \in [n]$.
Then by the symmetry assumption we have $I_{f,\br(h)}(j) \geq m^{-1/2}$
for at least $m$ values of $j$, so that 
 (using that $m \geq \qmax^{-9}$) we have
 $I_{f,\br(h)} \geq m^{1/2} \geq m^{1/3} \qmax^{-1.5}$, and
since $\log m \leq 3 m^{1/3}$,  this implies 
(\ref{0806c}).

Now suppose instead that
 $I_{f,\br(h)}(j) < m^{-1/2}$ for all $j \in [n]$.
Then since $m \geq \qmax^{-9}$ we have
 $I_{f,\br(h)}(j) < m^{-1/4} \qmax^{2}$ for all $j \in [n]$.
Setting
$a:= \max_{j \in [n]} I_{f,\br(h)}(j) /(\qmax^2 (\log (4/\qmax))^2)$,
we have that 
$a \leq m^{-1/4} \leq 1/16 $; also $\pmax ( \br(h)) \leq \qmax$, so by
Lemma \ref{leminfl} we have
$$
I_{f,\br(h)} \geq \frac{ g(h) (1-g(h)) \log (1/a) }{24
 k^2 \qmax \log (4/\qmax) } 
 \geq \frac{ g(h) (1-g(h)) \log m  }{96 k^2 \qmax \log (4/\qmax) } 
$$
which implies (\ref{0806c}).

For $0 \leq h \leq \gamma$ let $\tilde{g}(h) = \log ( g(h)/(1-g(h)))$. 
By (\ref{0806b}) and (\ref{0806c}) we have
\bean
\frac{d\tilde{g}}{d h }
= (g (1-g))^{-1} \frac{dg} {dh}
 \geq \frac{ \log m}{96 k^2  \qmax \log (4/\qmax) }.
\eean
Since $g(0) = \mathbb{P}_{\bp}^n(E)
\geq \eta$ by assumption, we have 
 $\tilde{g}(0) \geq \log \eta
= - \log (1/\eta)$, and using
the assumption (\ref{SharpBound}),  we obtain that
$$
\tilde{g}(\gamma)
 \geq - \log (1/\eta) +  \frac{ \gamma \log m}{96  k^2 \qmax \log (4/\qmax)}
\geq \log (1/\eta)
$$
which implies $g(\gamma) > 1- \eta$, and therefore also
$\mathbb{P}_{\bq}^n(E) > 1- \eta$.
\end{proof}

\section{Box crossings for eRSA}

We now return to eRSA.  
As mentioned in Section  \ref{secintro},
we shall apply Theorem \ref{SharpNM} using a discretization of
 time.
We shall compensate the error due to this,
by introducing a time-{\em delay} at the even sites.
In this section, we  therefore consider a version of eRSA
 where the arrivals at the even sites are slightly delayed.
We develop Margulis-Russo type
 formulae for the partial derivatives of box-crossing
probabilities  with respect to the parameters
$\lambda$, $p$ and the delay parameter,
 and estimates for the quantities arising from these formulae.
We shall use these later to prove
 Propositions \ref{cor:halfCorrection} and \ref{critequal}.

Given $\delta \geq 0$, we construct eRSA with
arrivals at even sites delayed by $\delta$
  from a collection of independent variables denoted $T_x$
 and $T_{x^\prime}$, defined for  $x\in\mathbb{Z}^2$.
 Here 
$T_x$ is exponential with parameter $1$
 for odd $x$ and with parameter $\lambda$ for even $x$, while
 $T_{x^\prime}$ is a uniform(0,1) random variable used to
determine whether the diamond site $x^\prime$ is black or white.  
The arrival time $t_x$ at $x$ is $t_x= T_x$ for odd $x$ and is
$t_x = T_x + \delta$ for even $x$.

Since the precise arrival times at the sites do not matter for the resulting
 distribution, merely the order of arrivals, we have the
 same jamming distribution if we move the arrival times at all octagon sites
 forward by amount $\delta$. Then by conditioning on the first
 arrival time at an odd site being at least $\delta$ and using
 the memoryless property of the exponential distribution we can arrive at the
same distribution if the arrival time $t_x$ at an even
 site is $T_x$ and at an odd site is $0$ with probability
 $1-e^{-\delta}$, otherwise taking the value $T_x$. 
Therefore we now assume that as well as the
variables $T_x $ and $T_{x'}$, $x \in \BZ^2$, we are provided with
uniform(0,1) random variables 
$U_x, x \in \mathbb{Z}^2 $.
For
 $x \in \mathbb{Z}^2$ we now set the arrival time 
$t_x$ to be $0$ if $x$ is odd and $U_x\leq 1-e^{-\delta}$; otherwise,
we set $t_x$ to be $T_x$. We set $x^\prime$ to be in the even phase
 if $T_{x^\prime}<p$ and in the odd phase otherwise.
Let $\BP_{\lambda,p,\delta}$ denote the resulting jamming distribution.


Given  $\rho \in (0,\infty)$ and
$n \in \N$ with $\rho n \geq 1$, let
 $H_{n, \rho}$ denote the event that there is a horizontal black crossing
 of the rectangle
 \[R(2n,\rho):=[-\lfloor \rho n \rfloor, \lfloor \rho n \rfloor-1]\times
[-n, n-1]
\] 
in the dependent face percolation model described
in Section \ref{secProof},
and set 
\[h_\rho(n,\lambda,p,\delta) := \BP_{\lambda,p,\delta}(H_{n, \rho}).\]

Next we introduce the concept of a site being {\em pivotal} for the event
$H_{n,\rho}$.
The definition will depend on
whether it is an odd site, an even site or a diamond site.

 We shall say that an odd site $x$ is pivotal for the event
 $H_{n, \rho}$ if $H_{n, \rho}$ occurs when we set the
 arrival time $t_x$ to $T_x$ but if we were to change the arrival
 time $t_x$ to $0$ (leaving all other variables constant), $H_{n, \rho}$ 
would no longer occur.

 We shall say that an even site $x$ is pivotal for
  the event $H_{n,\rho}$ if this event occurs when the arrival time at the
 site $x$ is $T_x$, but does not occur if we delay the arrival time at 
$x$ by an independent exponential random variable with rate 
$\lambda$ called $T$.

 For $y\in\mathbb{Z}^2$, we say that the diamond site
 $y^\prime$ is pivotal for the event $H_{n, \rho}$ if $H_{n, \rho}$
 occurs when $y^\prime$ is black but does not occur when 
$y^\prime$ is white.

 For any octagon or diamond site $z$ we define
\begin{align*}
\phi_{\lambda,p,\delta,\rho}(n,z)&:=\BP_{\lambda,p,\delta}[z \text{ is 
pivotal for event }H_{n,\rho}]
\end{align*}
\begin{proposition}\label{MR}
For any $\lambda, n,p$ and $\rho$, and for any $\delta \geq 0$
 it is the case that
\begin{equation} \label{eq:MRa} 
\frac{\partial h_\rho(n,\lambda,p,\delta)}{\partial p} 
= \sum_{x \in \mathbb{Z}^2}\phi_{\lambda,p,\delta,\rho}(n,x^\prime),
\end{equation}
\begin{equation}
\label{eq:LambdaPartDeriv}
\frac{\partial h_\rho(n,\lambda,p,\delta)}{\partial \lambda} = \lambda^{-1} \sum_{x \in \mathbb{Z}^2: x\text{ even}} \phi_{\lambda, p, 0,\rho}(n, x),
\end{equation}
and
\begin{equation} 
\label{eq:MRb} \frac{\partial h_\rho(n,\lambda,p,\delta)}{\partial \delta} 
= -e^{-\delta} \sum_{x\in \mathbb{Z}^2:x\text{ odd}}
 \phi_{\lambda,p,\delta,\rho}(n,x),
\end{equation}
where the partial derivative at $\delta =0$ is interpreted
as a one-sided right derivative. 
\end{proposition}

\begin{proof}
Equations \eqref{eq:MRa} and  \eqref{eq:LambdaPartDeriv} 
are as in Proposition 4.1 of \cite{RSAPaper}, and the proof there
 translates directly to this model.


For \eqref{eq:MRb}, fix $n, p, \lambda, \delta$. Enumerate the odd sites of $\mathbb{Z}^2$ in some manner as $x_1,x_2,\ldots$. Given $k \in \mathbb{N}$ and $\varepsilon >0$, let $\BP_{\delta, k, \delta+\varepsilon}$ denote the
probability measure
for a model with enhancement parameter $p$ and with arrival times
$t_x, x \in \BZ^2$, defined as follows. Let
the variables 
 $T_x,T_{x'},U_x,$ $x \in \BZ^2$, be as before.
Set $t_x = T_x$ for even $x$, and set
$$
t_{x} = \begin{cases}
 0 & \mbox{ if }  U_x < 1-e^{-\delta} \\ 
 T_x & \mbox{ otherwise, }
\end{cases} ~~~~~~ x \in \{x_1,\ldots,x_{k-1}\}
$$
and
$$
t_{x} = \begin{cases}
 0 & \mbox{ if }  U_x < 1-e^{-\delta - \eps} \\ 
 T_x & \mbox{ otherwise, }
\end{cases} ~~~~~~ x \in \{x_k,x_{k+1},x_{k+2}, \ldots\}.
$$

For $y, z \in \BZ^2$ let the notion of `$y$ {\em even-affects} $z$'
be defined in the same manner as `$y$ affects $z$', but in terms of the
arrival times at the {\em even} sites in a path from $y$ to $z$ being
in increasing order, rather than the odd sites. 
Let $A(x)$ be the event that the site $x$ even-affects
 some site in $R(2n,\rho)$. Then
\begin{align*}
0 \leq 
h_\rho(n,\lambda,p,\delta)
-
\BP_{\delta,k,\delta+\varepsilon}[H_{n,\rho}]
\leq \BP_{\delta,k,\delta+\varepsilon}[\cup_{j=k}^\infty A(x_j)] &\\
\to 0 \text{ as } k \to \infty, &
\end{align*}
 by (\ref{eqaffects}) with the word `affects' replaced by `even-affects',
 which is applicable since the arrival rates
at all {\em even} sites are the same. 
Thus,
\bean
h_\rho(n,\lambda,p,\delta+\varepsilon) - h_\rho(n,\lambda,p,\delta) & = &
 \BP_{\delta,1,\delta+\varepsilon}[H_{n,\rho}] - \lim_{k\to\infty} \BP_{\delta,k,\delta+\varepsilon}[H_{n,\rho}] 
\\
 & = & \sum_{k=1}^\infty(\BP_{\delta,k,\delta+\varepsilon}[H_{n,\rho}]-\BP_{\delta,k+1,\delta+\varepsilon}[H_{n,\rho}]). 
\label{sumLimit}
\eean

Given $\delta^\prime>0$ we now define $F_k(\delta,\delta^\prime)$ to be the event that $H_{n,\rho}$ occurs if we take the arrival time at site $x_k$ to be $T_{x_k}$, but not if we take it to be $0$, where an odd site $x_j$ has arrival time $0$ with probability $1-e^{-\delta}$ if $j<k$ and with probability $1-e^{-\delta^\prime}$ if $j>k$, otherwise having as arrival time $T_{x_j}$ (the dependence of $F$ on $p,\lambda$ and $n$ is suppressed).
With the variables as described above, we see that 
\begin{equation}
\label{pivResInequal}
\BP_{\delta,k,\delta+\varepsilon}[H_{n,\rho}]-
\BP_{\delta,k+1,\delta+\varepsilon}[H_{n,\rho}]=
-e^{-\delta}(1-e^{-\varepsilon})\BP[F_k(\delta,\delta+\varepsilon)].
\end{equation}

Couple $F_k(\delta,\delta+\varepsilon)$ and $F_k(\delta,\delta)$ by fixing
 the collection of random variables $U_{x_j}$ for $j\in\mathbb{N}$. 
For $K \in \N$ let $B(2K+1) := 
\left[-K,K\right]\times\left[-K,K\right]$.
Then for any integer $K>3 n$ we see that 
\begin{align*}
F_k(\delta,\delta+\varepsilon)\triangle F_k(\delta,\delta) \subset (&\cup_{x\in\mathbb{Z}^2\setminus B(2K+1)}A(x))\\
&\cup(\cup_{\{j> k : x_j\in B(2K+1)\}}\{1-e^{-\delta} < U_{x_j} < 1-e^{-(\delta+\varepsilon)}\}).
\end{align*}

For any fixed $K$, the probability of the event \[\cup_{\{j> k : x_j\in B(2K+1)\}}\{1-e^{-\delta} < U_{x_j} < 1-e^{-(\delta+\varepsilon)}\}\] vanishes as $\varepsilon\downarrow0$, and the probability of the event $\cup_{x\in\mathbb{Z}^2\setminus B(2K+1)}A(x)$ is independent of $\varepsilon$ and
 vanishes as $K\to\infty$. Then \eqref{pivResInequal} yields
\begin{align}
\label{limitIneq}
\lim_{\varepsilon\downarrow 0} \varepsilon^{-1} 
\left(\BP_{\delta,k,\delta+\varepsilon}\left[H_{n,\rho}\right] 
-\BP_{\delta,k+1,\delta+\varepsilon}\left[H_{n,\rho}\right]\right)
 &=-e^{-\delta}
\BP [F_k(\delta,\delta)] 
\\\notag
&=-e^{-\delta}\phi_{\lambda,p,\delta}(n,x_k).
\end{align}

Finally, using Lemma \ref{affectlem},
 note that $\BP [ F_k(\delta,\delta+\varepsilon) ]$ is
 bounded by $\BP[A(x_k)]$, which is
independent of $\varepsilon$
 and 
 summable in $k$ by (\ref{eqaffects}),
 so by \eqref{sumLimit}, \eqref{limitIneq}, and the dominated convergence
 theorem we have 
\begin{equation*}
\frac{\partial^+h_\rho}{\partial\delta}=\lim_{\varepsilon\downarrow0}
\frac{h_\rho(n,\lambda,p,\delta+\varepsilon)-
h_\rho(n,\lambda,p,\delta)}{\varepsilon}
=-e^{-\delta}\sum_{k=1}^\infty\phi_{\lambda,p,\delta,\rho}(n,x_k).
\end{equation*}
Provided $\delta >0$, a similar argument can be used to
 produce the same expression for the left partial derivative.
\end{proof}

Note that for any $q>0$ and fixed $n$, using
(\ref{eqaffects}) we can find a distance $r$ such that the probability that there exists a site at distance more than $r$ from $R(2n,\rho)$ that affects some site within $R(2n,\rho)$ is less than $q$. Hence for $\delta$ sufficiently small we have
 \begin{equation*}
\BP_{\lambda,p,0}(H_{n,\rho})-\BP_{\lambda,p,\delta}(H_{n,\rho}) 
\leq q+(1-e^{-\delta}) (6 n+2r)^2\leq 2q,
\end{equation*}
and thus $\BP_{\lambda,p,\delta}(H_{n,\rho})$ is right continuous in $\delta$ at $\delta=0$.

We seek to bound the effect of a slight change in $\delta$
in terms of the effect of a change in $p$. To do this 
 we shall use a variant of arguments
from \cite{RSAPaper}. 
Let $y$ be an odd site and let $r\in\mathbb{N}$;
 then define $C_r=C_r(y)$ be the square of side length $2r+1$ centred at $y$. Define $E_\rho(n,y,r)$ to be the event that if we use $t_y=0$ then 
(i) the event $H_{n,\rho}$ occurs if we change the colour of all the sites
 in $C_r$ to black (and leaving other sites unchanged) and 
(ii) the event $H_{n,\rho}$ does not occur if we change the colour of
 all the sites in $C_r$ to white.

\begin{lemma}\label{5.2}
Let $\eps \in (0,1/2)$, $\rho \in \N$.
There exists a constant $c_2 = c_2(\eps,\rho) \in \R_+ $ 
such that for any odd $y \in \BZ^2$,
any $n\in\mathbb{N}$, and
\begin{equation*} 
(\lambda,p,\delta)\in[\eps,1/\eps]\times[\eps,1-\eps]
\times[0,1]
\end{equation*}
we have 
\begin{equation*} 
\phi_{\lambda,p,\delta,\rho}(n,y) 
\leq \BP_{\lambda,p,\delta}[E_\rho(n,y,1)]+
\sum_{r=1}^\infty 
\frac{c_2^r \BP_{\lambda,p,\delta}[E_\rho(n,y,r+1)]}{\left\lfloor 
r/2\right\rfloor!}.
\end{equation*}
\end{lemma}

\begin{proof} 
The proof is similar to that of Lemma 5.2 of \cite{RSAPaper}.
For $r \in \N$
let $\tilde{E}_{\rho}(n,y,r)$ be
the event that (i) $y$ is pivotal for event $H_{n,\rho}$ and
(ii) event $H_{n,\rho}$ occurs when we use the arrival time
$t_y=0$ but then change all sites in $C_r$ to black.
As in  \cite{RSAPaper}, it is sufficient to prove that
there is a constant $c_2$ such that
$$
\leq \BP_{\lambda,p,\delta} 
[\tilde{E}_{\rho}(n,y,r+1) \setminus 
\tilde{E}_{\rho}(n,y,r)] 
\leq 
\frac{c_2^r \BP_{\lambda,p,\delta}[E_\rho(n,y,r+1)]}{\left\lfloor 
r/2\right\rfloor!}, ~~~ r \in \N.
$$
This is proved in the same manner as the corresponding equation
(5.6) of \cite{RSAPaper}. In short, the idea is to define
the event $F(r)$ that $y$ affects some site outside $C(r)$; to observe that
$$
\tilde{E}_{\rho}(n,y,r+1) \setminus 
\tilde{E}_{\rho}(n,y,r) \subset F(r) \cap E_{\rho}(n,y,r+1),
$$
and then to use a coupling device 
to show that 
there is a constant $c$ such that
$$
 \BP_{\lambda,p,\delta} [ 
{E}_{\rho}(n,y,r+1) \cap  F(r)  ] 
\leq c^r 
\BP_{\lambda,p,\delta} [ 
{E}_{\rho}(n,y,r+1)]
\BP_{\lambda,p,\delta,\rho} [    F(r)  ]. 
$$
For details, see \cite{RSAPaper}.

\end{proof}


For 
$y\in\mathbb{Z}^2$, let $z_\rho(n,y)$ be the nearest even site in
 $R(2(n-3),\rho)$ to $y$ using Euclidean distance, taking the
 first according to the lexicographic ordering when there is a choice. 
Let $z_\rho^\prime(n,y):=z_\rho(n,y)+(1/2,1/2).$

\begin{lemma}\label{Lem5.1}
Let $\eps \in (0,1/2)$,  $\rho\in\mathbb{N}$. 
There exists a constant $c_3=c_3(\eps,\rho) \in (0,\infty)$ such 
that for any odd $y \in \BZ^2$, any $n\in\mathbb{N}$ with 
$n\geq 60$, and  any
$
 (\lambda,p,\delta,r)\in[\eps,1/\eps]\times[\eps,1-\eps]
\times[0,1]\times\mathbb{N},
$
we have that
\begin{equation*}
\BP_{\lambda,p,\delta}[E_\rho(n,y,r)]\leq c_3^r
 \phi_{\lambda,p,\delta}(n,z_\rho^\prime(n,y))
\1_{R(2(n+r),\rho)}(y).
\end{equation*}
\end{lemma}
\begin{proof}
This is proved in the same manner as
Proposition 5.1 of \cite{RSAPaper}. The only differences compared with that
result is that here we consider crossing of the rectangle $R(2n,\rho)$
whereas in \cite{RSAPaper} it was the square $R(2n,1)$, and that here
$y$ is odd whereas in \cite{RSAPaper} $y$ is even. These have little effect
on the argument.

The idea of the argument is as follows.
Suppose $y$ is such that $B(2r)$ is contained in $R(2n,\rho)$
(the other case is considered separately but the argument
is not dissimilar in that case).
 If the event $E_{\rho}(n,y,r)$ occurs then there exist
disjoint black paths from the left and right sides of $R(2n,\rho)$
to the boundary of $B(2r)$. One can establish existence of a collection of
$O(r)$   sites in $B(2r)$, such that
if we resample the arrival times and enhancement variables inside
$B(2r)$ (but change nothing outside $B(2r)$),  then given
conditions on the resampled outcomes at this set of sites 
we will have $z'_\rho(n,y)$ being pivotal.
For further (quite lengthy) details, see
\cite{RSAPaper}. 
%
\end{proof}

\begin{lemma}
\label{pivInequal}
For any  $\eps \in (0,1/2)$, $\rho \in \N$,
there is a constant $c_4=c_4(\eps,\rho)$ such that
for any odd $y$, $n\in\mathbb{N}$, and
$
(\lambda,p,\delta)\in[\eps,1/\eps]\times[\eps,1-\eps]\times[0,1],
$
we have
 \begin{equation*}
\sum_{y\in\mathbb{Z}^2:y ~ \text{{\rm odd}}}
\phi_{\lambda,p,\delta,\rho}(n,y) 
\leq c_4 \sum_{z\in R(2(n-3),\rho):z ~ \text{{\rm even}}}
\phi_{\lambda,p,\delta,\rho}(n,z^\prime).
\end{equation*}
\end{lemma}
\begin{proof}
Using our Lemmas \ref{5.2} and \ref{Lem5.1}, the proof is
 as in the first step
of the proof of Proposition 3.1 in \cite{RSAPaper}.
 In this case, this step is just a few lines; one groups the
terms in the sum on the left according to those $y$ for
which $z_\rho(y) $ takes the value of a particular term
$z$ in the sum on the right.
\end{proof}
\begin{corollary}\label{cor:CompensateEqual}
For any $\eps \in (0,1/2)$, $\rho \in \N$
 there exists a constant $c_5=c_5(\eps,\rho)$ such that for any
$n\in\mathbb{N}$, and
$ (\lambda,p,\delta)\in[\eps,1/\eps]\times[\eps,1-\eps]\times[0,1], $
we have
\begin{equation}
\label{CompensateIneq}
\left|\frac{\partial h_\rho(n,\lambda,p,\delta)}{\partial \delta}\right| 
\leq c_5\frac{\partial h_\rho(n,\lambda,p,\delta)}{\partial p}.
\end{equation}
\end{corollary}
\begin{proof}The result follows immediately from Lemma \ref{pivInequal} and Proposition \ref{MR}.
\end{proof}

\section{Proof of Proposition \ref{cor:halfCorrection}} 
\label{secKeys}

Proposition \ref{cor:halfCorrection} says that
the effect on the crossing probability $h_\rho( n, \lambda,p, \delta)$
of a small change in $\lambda$, is comparable to the effect
of a small change in $p$. 
To prove this, we need to find an appropriate inequality connecting even
 sites being pivotal, and diamond sites being pivotal.
 Figure \ref{fig:diaPiv} demonstrates one of the four possible arrangements of occupied sites closest to the diamond site in question (the other possibilities
 being the reflection of the occupation locations and colour inversions of
 these two). In order for this diamond site to be pivotal, in addition to
 the sites locally having an arrangement of this form we also require
 that there be a black path from the left edge of the rectangle to one
 of the occupied black sites close to the diamond site, a black path from
 the right edge to the other occupied black site, a white path from
 the top edge to one of the occupied white sites, and a white path
 from the bottom edge to the other occupied white site.

\begin{figure}[htbp]
	\centering
		\includegraphics{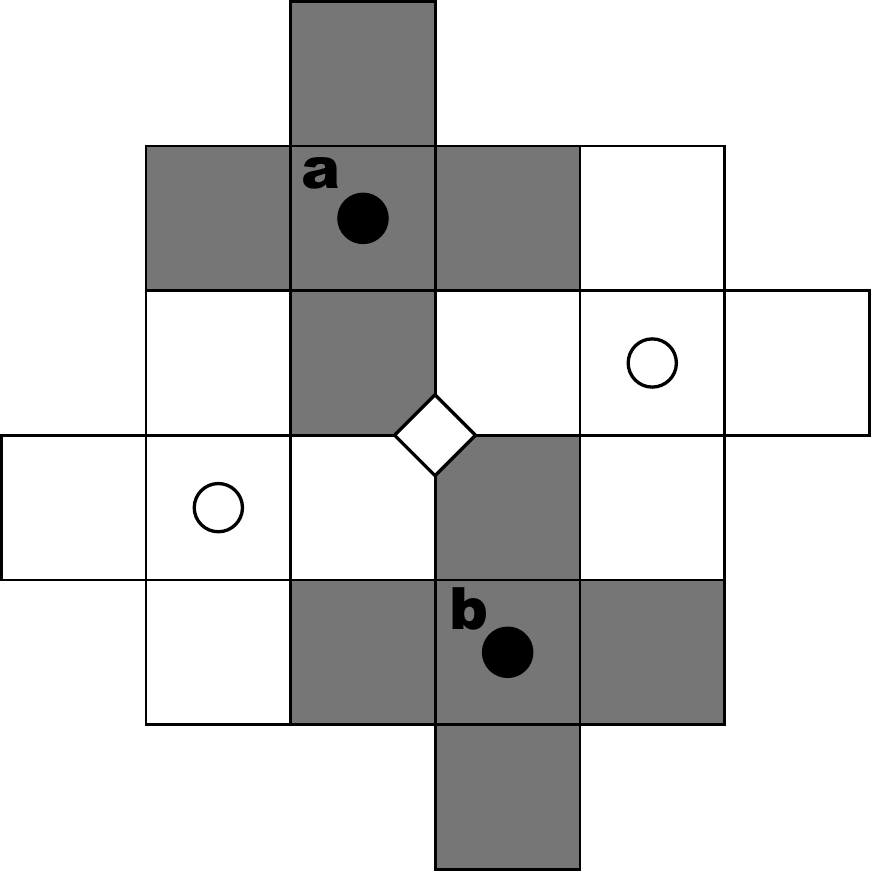}
		\caption{An example of the possible local arrangements of occupied and blocked sites such that a diamond site may be pivotal.}
	\label{fig:diaPiv}
\end{figure}

Recall the definition  that for any octagon site $y$, the
 site $y^\prime$ is the site $y+(1/2,1/2)$, and define similarly 
$y^{\prime\prime}$ as the site $y+(1/2,-1/2)$.

\begin{lemma}
\label{thm:diaPivToEvenPiv} 
For any $\varepsilon \in (0,1)$ there is a constant $c_6=c_6(\varepsilon)>0$,
 such that for any $\lambda \in [\eps, 1/\eps]$,
 and $p \in [0, 1]$,
$n\in\mathbb{N}$, and any even  $y \in \BZ^2$ we have
\begin{equation}
\label{0831a} 
\phi_{\lambda, p, 0,3}(n, y^\prime)
 \leq c_6 \phi_{\lambda, p, 0,3}(n, y),
\end{equation}
and
\begin{equation} 
\phi_{\lambda, p, 0,3}(n, y^{\prime\prime}) \leq c_6
 \phi_{\lambda, p, 0,3}(n, y).
\label{0813b}
\end{equation}
\end{lemma}

\begin{proof}
Fix an even site $y$, and let $\mathcal{S}_y = (S_x)_{x\in\Lambda}$ be
 the collection of arrival times and enhancement variables in one 
eRSA process. In a similar manner to the proof of Proposition 5.1 of
 \cite{RSAPaper}, we shall construct a coupled process 
$\mathcal{U}_y = (U_x)_{x\in\Lambda}$.
For $n \in \N$,
let $B(2n+1)$ be the collection of octagon and diamond sites within
 $[-n,n] \times [-n,n]$. Let $\mathcal{S}_y$ be as above, let $\mathcal{T}_y = (T_x)_{x\in\Lambda}$ be the set of arrival times and enhancement variables in an independent RSA process, and let $\mathcal{B} = (B_x)_{x\in\Lambda}$ be a collection of 
independent Bernoulli random variables with parameter $0.5$. Then we define
\bean
U_x &= S_x,									 & x \in \BZ^2 \setminus( B(13)+y); \\
		&= T_x,									 & x \in \BZ^2 \cap(  B(7)+y);  \\
		&= B_x T_x + (1-B_x) S_x, & x
\in \BZ^2 \cap ((B(13) \setminus B(7)) +y); 
\\
& = S_x,  & x - (1/2,1/2) \in \BZ^2.
\eean
We also define an independent exponential random variable $T$ with parameter $\lambda$.

We now define three events denoted $E_1$, $E_2$ and $E_3$,
 such that if all three events hold, then the site $y$ is pivotal
 in the $\mathcal{U}_y$ process. Let $E_1$ be the event that the
 diamond site $y^\prime$ is pivotal for the $\mathcal{S}_y$ process. 
For $m < n$, let $A_y(m,n)$ be the square annulus $y+B(n)\setminus B(m)$. 
We shall define $E_2$ to be an event
concerning sites in the annulus $A_y(3,13)$ which ensures that for
the ${\cal U}_y$ process the occupied octagon sites of the
${\cal S}_y$ process therein 
have an earlier arrival time than all of their neighbours
in that annulus, and moreover all occupied sites in $A_y(3,7)$ have
arrival times between $0.5$ and $1$.
 Define $E_2$ as follows:
 \begin{align*}
E_2 = \bigcap&\cap_{\{x\in A_y(7,13)\cap\mathbb{Z}^2 : S_x>1
 \text{ and } x \text{ is occupied in } \mathcal{S}_y\}}\{B_x = 1
 \text{ and } T_x<1\}\\
&\cap_{\{x\in A_y(7,13)\cap\mathbb{Z}^2 : S_x<1 \text{ and } x
 \text{ is blocked in } \mathcal{S}_y\}}\{B_x = 1 \text{ and } T_x>1\}\\
&\cap_{\{x\in A_y(7,13)\cap\mathbb{Z}^2 : S_x<1 \text{ and } x \text{ is occupied in } \mathcal{S}_y\}}\{B_x = 0\}\\
&\cap_{\{x\in A_y(7,13)\cap\mathbb{Z}^2 : S_x>1 \text{ and } x \text{ is blocked in } \mathcal{S}_y\}}\{B_x = 0\}\\
&\cap_{\{x\in A_y(3,7)\cap\mathbb{Z}^2 : x \text{ is occupied in } \mathcal{S}_y\}}\{0.5<T_x<1\}\\
&\cap_{\{x\in A_y(3,7)\cap\mathbb{Z}^2 : x \text{ is blocked in } \mathcal{S}_y\}}\{T_x>1\}.
\end{align*}
We shall define $E_3$ to be an event concerning 
the sites in $y+B(3)$ which ensures  (in conjunction with $E_2$)
that the sites next to $y$ will become occupied if the arrival at
$y$ is delayed but blocked if the arrival at $y$ is not delayed. 
 To be precise, define 
\begin{align*}
E_3:=\{T_y\leq 0.1\} &\cap_{z\in y+B(3): z\,\mathrm{ odd}} \{0.1< T_z\leq 0.2\}\cap 
\\&\cap_{z\in y+B(3): z\,\mathrm{ even}, z \neq y}
 \{0.2< T_z \leq 0.3\} \cap \{T > 0.2\}.
\end{align*}

Consider the state of the $\mathcal{U}_y$ process if all of
 these events occur. If $E_2$ and $E_3$ both hold,
 then every even octagon site within the square $y+B(3)$ is occupied
 if we have the arrival time at $y$ being $T_y$, but blocked if
 we delay the arrival at $y$ by $T$. As noted in Lemma 5.1 of \cite{RSAPaper},
 provided $E_2$ occurs then the states of sites outside $y+B(7)$
 in the $\mathcal{U}_y$ process match the states of those sites in the 
$\mathcal{S}_y$ process. Now we consider any even octagon site within 
$A_y(3,7)$.
 If this site was black in the $\mathcal{S}_y$ process, then in the
 $\mathcal{U}_y$ process it has arrival time less
 than $1$ and any adjacent sites outside $y+B(3)$ have arrival
 times at least $1$, thus are unable to block it. Since all odd
 sites within $y+B(3)$ are blocked by the arrival at $y$, it follows
 that the site under consideration has first arrival time strictly lower 
than all adjacent unblocked sites and hence is occupied.

Suppose $y'$ is pivotal in $\mathcal{S}_y$.
Without loss of generality, we assume that in the $\mathcal{S}_y$ process the local arrangement of occupied sites at $y^\prime$ matches that in figure \ref{fig:diaPiv} and that the site labelled $a$ has a black path connecting it to the left side of the rectangle, and that the site labelled $b$ has a black path connecting it to the right side of the rectangle. By our argument and due to black paths being increasing in black sites, it follows that in the $U_y$ process there is a black path from the left side of the rectangle to $a$, from the site $a$ to the site $b$ due to all the sites in the square $y+B(3)$ being black, and from the site $b$ to the right side of the rectangle. As such, we see that in the $\mathcal{U}_y$ process, if the events $E_1$, $E_2$ and $E_3$ hold and we take $T_y$ as the arrival time at $y$ we have a horizontal black crossing of the rectangle.

A similar argument shows that if we delay the arrival at $y$ by the random variable $T$ and the events $E_1$, $E_2$ and $E_3$ hold then we have a vertical white crossing of the rectangle, and thus 
the site $y$ is pivotal. We then obtain (\ref{0831a})
 by noting that the events $E_1 \cap E_2$ and $E_3$ are independent,
 that the probability of $E_1$ is $\phi_{\lambda, p, 0,3}(n, y^\prime)$, 
and that there is a strictly positive lower bound
both on $\BP_{\lambda,p,0}[E_2|E_1]$ and on $\BP_{\lambda,p,0}[E_3]$,
 uniformly over $0 \leq p \leq 1$ and 
$\eps \leq \lambda \leq 1/\eps$, and over outcomes of the ${\cal S}_y$ 
process in event $E_1$.

A similar argument provides the second inequality
(\ref{0813b}).
\end{proof}


\begin{proof}[Proof of Proposition \ref{cor:halfCorrection}]
 The second inequality 
of (\ref{eq:LambdaComp})
follows immediately from
Lemma \ref{thm:diaPivToEvenPiv}, 
 equation \eqref{eq:LambdaPartDeriv} and 
equation \eqref{eq:MRa}.

The first inequality 
of (\ref{eq:LambdaComp})
is obtained
as in the proof of Proposition 3.1 of \cite{RSAPaper}.
\end{proof}

{\em Remark.} 
The proof of Lemma \ref{thm:diaPivToEvenPiv} 
 (and hence, of the second inequality of 
(\ref{eq:LambdaComp})) is simpler than the
the proof of Proposition 3.1 of \cite{RSAPaper}, required
for the proof of the first inequality of
(\ref{eq:LambdaComp}). This is because in proving Lemma  
\ref{thm:diaPivToEvenPiv} we change a configuration with
a pivotal diamond site so  that a neighbouring octagon site
is pivotal, and can arrange that changing the arrival time
at the octagon site affects the nearby sites in
a manner which helps
to make it pivotal. For the inequality the other way,
we need to change a configuration with a pivotal octagon
site to make a neighbouring diamond site pivotal, which
is more complicated since the diamond site has no
effect on other sites, so we need to change the configuration
of states of nearby octagon sites `by hand' to make the diamond site pivotal.

\section{Proof of Proposition \ref{critequal}}
\label{seckey2}
To prove Proposition \ref{critequal} we shall use
Proposition \ref{SharpNM}, our sharp thresholds result.
Since that result refers to a discrete product space,
 we shall need to discretise
time, and also transfer the model to a torus to
achieve the symmetry needed for applying 
Proposition \ref{SharpNM}.

Given $n \in \N$, let $\mathbb{T}(2n)$ denote the torus formed
 from a $2n$ by $2n$ square of octagon sites and the diamond sites at the upper
 right corner of each octagon site. We shall arbitrarily choose an octagon
 site in the torus to be the origin, and from this we can define even
 and odd sites on $\mathbb{T}(2n)$ and hence have enhanced RSA as before
 on the torus. Where required, we shall denote by
 $P^{\mathbb{T}(2n)}_{\lambda,p}$ and $P^{\Lambda}_{\lambda,p}$ the 
probability measures for enhanced RSA with parameters $\lambda$ and $p$
 on the torus $\mathbb{T}(2n)$ and on the full enhanced integer
 lattice $\Lambda$ respectively.

\begin{lemma}
\label{TorApprox}
Let $n\in\mathbb{N}$, $\lambda>0$ and $p\in(0,1)$, and let $R$
 be a rectangle with long side length at most $2n-4\sqrt{2n}$. Then
 \begin{equation*}
\left|P^{\mathbb{T}(2n)}_{\lambda,p}\left[H(R)\right]-P^\Lambda_{\lambda,p}\left[H(R)\right]\right|<e(n)
\end{equation*}
where $H(R)$ is the event that $R$ has a horizontal black crossing and $e(n)$ is some $o(1)$ function independent of $R$.
\end{lemma}
\begin{proof}
We can couple enhanced RSA on $\Lambda$ and on $\mathbb{T}(2n)$ such that
 the arrival times at integer sites and colours of diamond sites
 agree on $\{(a,b):0\leq a,b\leq2n\}$. By Lemmas \ref{affectlem}
and \ref{3.3} the
 probability that there is a site within a rectangle contained within
 $\{(a,b):2\left\lceil\sqrt{2n}\right\rceil
 \leq a,b\leq2n-2\left\lceil\sqrt{2n}\right\rceil\}$ 
whose colour disagrees with the colour of the associated site in $\Lambda$ tends to $0$ as $n \to \infty$, and so the result follows.
\end{proof}

Given $n \in \N$,
define the torus $\mathbb{T}=\mathbb{T}(20n)$. 
Let $Q_n$ be the $20n$ by $20n$ square region in $\Lambda$ which is identified
 with $\mathbb{T}$.
Also let $Q_n^e$ be the set of even sites in $Q_n$ and
let $\mathbb{T}^e$ be the set of even sites in $\mathbb{T}$. 
Set $\delta := \delta(n) :=(\log n)^{-1/2}$.

Given $(\lambda_0,\tilde{p},\lambda_1)\in\mathbb{R}_+\times
(0,1) \times \mathbb{R}_+ $,
let $\BP_{\lambda_0,\tilde{p},\lambda_1}^{\mathbb{T}}$ be
 the probability measure 
associated with the enhanced RSA model on the torus $\mathbb{T}$ with
 arrivals rate $\lambda_0$ 
at even sites and $\lambda_1$ at odd sites,
and diamond sites black with probability $\tilde{p}$. 
(When $\lambda_1=1$, we sometimes omit the 
 third subscript $\lambda_1$  
from the notation.)
We now construct a discrete-time version of this process.
At each site $x\in\mathbb{Z}^2$ we shall divide the time-axis into blocks
 of length $\delta$, and discarding all blocks that had their start time 
later than $n$ we have a product space 
$\tT := \mathbb{T} \times \{-1,0,1,2,\ldots,\lfloor n/\delta\rfloor\}$ 
where $(x,-1)$ represents the diamond site $x^\prime:=x+(1/2,1/2)$, and
 $(x,k)$ for $k\in\{0,1,\ldots,\lfloor n/\delta\rfloor\}$ represents the site
 $x$ at times in the interval $I_\delta(k) := [k\delta,(k+1)\delta)$. We 
denote the probability measure on this new space by 
$\BP_{\lambda_0,\tilde{p},\lambda_1}^{\tT}$.

We shall now construct a random field
 $\ubx=(X(x,k):(x,k)\in\tT )$ 
with each $X(x,k)$ taking values in
 $\{0,1,2,3\}$. For an even site $x$ and for $k \geq 0$, we
 set $X(x,k)=3$ if there is an attempted arrival at $x$ within
 $I_\delta(k)$ and $X(x,k)\in\{0,1,2\}$ if not.
 For an odd site $x$ and for $k\geq 0$, we set
 $X(x,k)=0$ if there is an attempted arrival at $x$ within $I_\delta(k)$ and
 $X(x,k)\in\{1,2,3\}$ if not.
For 
any site $x$, we put $X(x,-1)\in\{2,3\}$ if $x^\prime$ is black, and
 $X(x,-1)\in\{0,1\}$ if $x^\prime$ is white. To construct a representation of 
this model in a discrete product space we consider all arrivals at a site 
instead of solely the first, so that $X(x,k_1)$ is independent of 
$X(x,k_2)$ whenever $k_1\neq k_2$. Where there is a choice of
 the value of $X(z)$ for
 $z\in\tT$, 
we choose randomly and 
independently of $X(z^\prime)$ for all $z^\prime\neq z$ so that the
 distribution of ${\bf X}$, denoted
$\BP_{\lambda_0,\tilde{p},\lambda_1}^{\tT}$, 
satisfies

\begin{align}
\BP_{\lambda_0,\tilde{p},\lambda_1}^{\tT} 
(X(z)=3)
 &= 1-e^{-\lambda_0\delta};
\label{discreteProbBlock}
\\
\BP_{\lambda_0,\tilde{p},\lambda_1}^{\tT} (X(z)=2) 
&= \tilde{p}+e^{-\lambda_0\delta}-1;\notag\\
\BP_{\lambda_0,\tilde{p},\lambda_1}^{\tT} (X(z)=1) 
&= e^{-\lambda_1\delta}-\tilde{p};\notag\\
\BP_{\lambda_0,\tilde{p},\lambda_1}^{\tT} (X(z)=0) 
&= 1-e^{-\lambda_1\delta}.\notag
\end{align}
Since we assume $\tilde{p} \in (0,1)$, for large enough $n$ these really are
probabilities.


 Let $\efast(\BT^e)$ be the event that
 for all $z\in \BT^e $
 the first arrival time at $z$ is less than $\sqrt{n}$,
and let
  $\efast(Q_n^e)$ be defined similarly.

Let $E_n$ be the event that there is some $18n$ by $2n$ rectangle in
 $\mathbb{T}$ with a horizontal black crossing after the arrival times at 
all even sites are delayed by $2\delta$ and that also $\efast(\mathbb{T}^e)$
 occurs.
Let $E_n^\text{crude}$ be the event that the state
 of $\ubx:=\{X(x,k):(x,k)\in\tT\}$ 
is such that $E_n$ is possible given $\ubx$; this can be seen as either
 an event on the discrete time torus
$\tT$,
 or as an event on the continuous time torus representing that
 the state of $\ubx$ consistent with the arrival times satisfies the 
understanding of $E_n^\text{crude}$ above.

\begin{lemma}
\label{Lem0613}
Let $\lambda >0$, $ p \in (0,1)$ and $\varepsilon \in (0, 1-p)$. 
  Suppose for some 
$\rho >0$
 that
 $\limsup_{n \to \infty} h_{\rho} (n,\lambda,p) >0$.   
Then
\begin{equation}
\limsup_{n \to \infty}
 \BP_{\lambda,p+\varepsilon/2,1}^{\tT} (E_n^{\text{\rm crude}}) >0. 
\label{discMin}
\end{equation}
\end{lemma}
\begin{proof}
Let $R_n := [1,18n] \times [1,2n]$, considered as a rectangle in 
 the torus $\mathbb{T}$. 
 Let $D_n$ be the event that $R_n$ has a horizontal black
 crossing after the arrival times at all even sites are delayed by 
$2\delta$.
By the union bound and the exponential decay of
the tail of the exponential distribution, we have
$
\lim_{n \to \infty}
 \BP^{\Lambda}_{\lambda,p+\eps/2}(\efast(Q_n^e))=1.
$
Hence using Lemma \ref{3.3},
  letting $\edense^n$ be the event 
$\edense(R_n,2\left\lceil\sqrt{2n}\,\right\rceil)$ 
 we have
\begin{align*}
\BP^\mathbb{T}_{\lambda,p +\varepsilon/2}\left[
D_n\cap \efast(\mathbb{T}^e)\right]
 &\geq \BP^\Lambda_{\lambda,p +\varepsilon/2}\left[
D_n \cap \edense^n \cap \efast(Q_n^e)\right] 
\\
&= \BP^\Lambda_{\lambda,p +\varepsilon/2}(D_n)+o(1).
\end{align*}
Then using Corollary \ref{cor:CompensateEqual} and the Mean Value Theorem, 
followed by Lemma \ref{LimResBetter},
 we obtain
\begin{align}
\limsup_{n \to \infty} \BP^\mathbb{T}_{\lambda,p +\varepsilon/2}
\left[D_n\cap \efast(\mathbb{T}^e)
\right]&\geq \limsup_{n \to \infty} \BP^\Lambda_{\lambda,p}(H_{n,9})
\notag\\
%
& >0.
\label{E1Ineq}
\end{align}

Clearly
 $\BP_{\lambda,p + \eps/2}^\mathbb{T}(E_n)\geq 
\BP_{\lambda,p + \eps/2}^\mathbb{T}[D_n
 \cap E_{\rm fast}^n(\mathbb{T}^e) ]$.
Since for any $(\lambda_0,\tilde{p},\lambda_1)$ we have 
$$
\BP_{\lambda_0,\tilde{p},\lambda_1}^{\tT}
(E_n^{\text{crude}})\geq
 \mathbb{E}\left[\BP_{\lambda_0,\tilde{p},\lambda_1}^{\mathbb{T}}
(E_n | {\bf X})\right]=
\BP_{\lambda_0,\tilde{p},\lambda_1}^{\mathbb{T}}(E_n),
$$
 from \eqref{E1Ineq} we have 
(\ref{discMin}).
\end{proof}

We now we use our sharp thresholds result to
show that after a slight adjustment of parameters,
the probability of the discrete event 
$E_n^{\text{crude}}$ is infinitely often close to 1 rather than just
being bounded away from zero as in (\ref{discMin}).

\begin{lemma}
\label{Lem0615}
Under the assumptions of Lemma \ref{Lem0613},
 \begin{equation}
\label{0615a}
\limsup_{n \to \infty}
 \BP_{(1+\eps)^{1/2}\lambda,p+\varepsilon/2,(1+\eps)^{-1/2}}^{\tT} 
(E_n^{\text{\rm crude}})  =1.
\end{equation}
\end{lemma}

\begin{proof}
Set $ N := N(n) := 400n^2(2+\left\lfloor\frac{n}{\delta}\right\rfloor)$. 
Given a probability vector 
$\bp' = (p^\prime_0,p^\prime_1,p^\prime_2,p^\prime_3)$,
 we define the probability measure
 $\BP^n_{\bp'} $
on  the space 
$\{0,1,2,3\}^N$
 as in Section \ref{secthresh}.
We can now think of $E_n^{\text{crude}}$ as being an event
 $E_n^{\text{disc}}$ in
$\{0,1,2,3\}^N$,
 by enumerating the $(x,k)$ pairs as $z_1,z_2,\ldots,z_N$ 
and identifying the value of ${\bf X}$ with an element of
$\{0,1,2,3\}^N$.
Given $\lambda_0,\tilde{p} $ and $\lambda_1$, the
distribution of ${\bf X}$ under this identification is given by
 $\BP^N_{\bp} $
with the entries of $\bp$  given by 
 (\ref{discreteProbBlock}).

The event $E_n$ is symmetric under the group of permutations of sites
 by translations of the torus (modulo $20n$) that send even sites to even
 sites, and therefore  so too are $E_n^{\text{crude}}$ and 
 $E_n^{\text{disc}}$. This group of permutations has order
$200n^2$.

We claim that $E_n^{\text{disc}}$ is increasing in $\ubx$. Indeed,
suppose $z=(x,k)\in\tT$.  If 
$k= -1$ then $z$ corresponds to the diamond site $x'$, and
 an increase in $X(z)$
 corresponds either
 to leaving $x'$ unchanged, or
 to changing $x'$ from being white  to being black.
If $k \geq 0$ and $x$ is an odd site,
 an increase in $X(z)$ from $0$ corresponds to removing any arrivals
at $x$ in the time period  $I_\delta(k)$ and otherwise leaving things
 unchanged. If $k \geq 0$ and $x$ is an even site, an increase in $X(z)$
 corresponds to either leaving things unchanged, or 
adding an arrival at $x$ in the time period
$I_\delta(k)$. Thus regardless of the nature 
of a site $z$, $E_n^{\text{disc}}$ is increasing in $X(z)$.

In order to apply Proposition \ref{SharpNM}, we compare two models,
i.e. two probability vectors $(p_0,p_1,p_2,p_3)$
and $(q_0,q_1,q_2,q_3)$,
 where $p_i$ is the probability that $X(z)=i$ in the first model,
 and $q_i$ is the probability that $X(z)=i$ in the second model.
 Our first model has parameters $\lambda_0= \lambda$, $\lambda_1=1$, and 
$\tilde{p}=p +\varepsilon/2$, while our second model has parameters
 $\lambda_0 = (1+\varepsilon)^{1/2} \lambda $, 
$\lambda_1=(1+\varepsilon)^{-1/2}$ and $\tilde{p}=p+\varepsilon$. Then
 using \eqref{discreteProbBlock} 
we have 
\begin{align*}
p_3&= 1-e^{-\lambda \delta}, & q_3&= 1-e^{-(1+\varepsilon)^{1/2}\lambda 
\delta};
\\
p_2&= e^{-\lambda \delta}+\varepsilon/2 + p -1 , & 
q_2&= e^{-(1+\varepsilon)^{1/2} \lambda \delta}+\varepsilon + p -1;\\
p_1&= e^{-\delta}-\varepsilon/2- p , &
 q_1&= e^{-(1+\varepsilon)^{-1/2}\delta}-\varepsilon- p;\\
p_0&= 1-e^{-\delta}, & q_0&= 1-e^{-(1+\varepsilon)^{-1/2}\delta}.
\end{align*}
From the equivalence of 
$E_n^{\text{disc}}$ and $E_n^{\text{crude}}$,
and Lemma \ref{Lem0613}, we have
$$
\limsup_{n \to \infty}
 \BP_{p_0,p_1,p_2,p_3}^N(E_n^{\text{disc}}) 
= \limsup_{n \to \infty}
\BP^{\tT}_{\lambda,p+\varepsilon/2,1}
(E_n^{\text{crude}}) > 0.
$$

We shall now apply Proposition \ref{SharpNM}. Note that 
\begin{align*}
q_3-p_3 &= e^{-\lambda \delta}-e^{-(1+\varepsilon)^{1/2}\lambda \delta}
 \sim
 \delta((1+\varepsilon)^{1/2}-1) \lambda;
\\
p_1-q_1 &= \varepsilon/2+e^{-\delta}-e^{-(1+\varepsilon)^{-1/2}\delta}
 \to \varepsilon/2 ; 
\\
p_0-q_0 &= e^{-(1+\varepsilon)^{-1/2}\delta}-e^{-\delta} \sim
 \delta(1-(1+\varepsilon)^{-1/2})
\\
&= \delta\frac{(1+\varepsilon)^{1/2}-1}{(1+\varepsilon)^{1/2}} . 
\notag
\end{align*}
 Set $\gamma = \min(p_0 - q_0 , q_3 -p_3)$.
For sufficiently high $n$, we obtain that 
$p_0 > q_0$, $p_1 > q_1$ and $q_3 > p_3$.
Hence $\gamma >0$ and $(q_0,q_1,q_2, q_3)$ dominates
$(p_0 - \gamma,p_1,p_2,p_3+\gamma)$.
We shall apply Proposition \ref{SharpNM} with $k=3$.
In the terminology of that result, we have $\qmax =\min(p_2,p_1)$.
Fix
\begin{equation*}
\eta \in 
(0, \limsup_{n \to \infty} \BP_{p_0,p_1,p_2,p_3}^N(E_n^{\text{disc}})). 
\end{equation*}
Since $p \log(4/p)$ takes maximum value $\log(4)<2$,
 the right hand side of \eqref{SharpBound} is at most 
$ 4000 \log (1/\eta) $.
Since $\delta = (\log n)^{-1/2}$,
 for $n$ large enough we have
$\gamma \log (200 n^2) > 4000 \log (1/\eta)$. 
 Thus  Proposition \ref{SharpNM} is applicable; by that result,
and the equivalence of 
$E_n^{\text{disc}}$ and $E_n^{\text{crude}}$,
 for infinitely many $n$ we have
 \begin{equation*}
\BP^{\tT}_{(1+\varepsilon)^{1/2} \lambda,
p +\varepsilon,(1+\varepsilon)^{-1/2}}
(E_n^{\text{crude}}) =
\BP^N_{q_0,q_1,q_2,q_3}(E_n^{\text{disc}})
>1-\eta,
\end{equation*}
and (\ref{0615a}) follows.
\end{proof}

\begin{lemma}
\label{highProbCrossing}
Let $\lambda >0$, $ p \in (0,1)$ and
 $\varepsilon \in (0, 1-p)$.
  Suppose for some $\rho >0$  that
 $\limsup_{n \to \infty} h_{\rho} (n,\lambda,p) >0$.   
Then 
 \begin{equation}
\label{mainResLimSup}
\limsup_{n\to\infty} h_3
(n,\lambda (1 + \varepsilon),p +\varepsilon) =1.
\end{equation}
\end{lemma}
Recall that our goal is to prove Proposition \ref{critequal},
which gives a similar conclusion for
 $ h_3(n,\lambda +\eps,p) $ and
 $h_3(n,\lambda ,p+\eps) $. Thus with this lemma, we are nearly there.

\begin{proof}[Proof of Lemma \ref{highProbCrossing}]
 Let $F_n$ be the event that there is a horizontal black 
crossing of some $18n$ by $2n$ rectangle in $\mathbb{T}$ (like $E_n$
but with no time delay and with no requirement for the
event $\efast(\BT^e)$ to occur). We assert
 the event inclusion
$ E_n^{\text{crude}} \subset F_n$.

Indeed, consider any state $\ubx_0 \in E_n^{\text{crude}}$.
 Let $x_0, x_1,...$ be an enumeration of the sites of 
$\mathbb{T} \cap \mathbb{Z}^2$, let $\text{col}_{x^\prime}$ denote
 the colour of the diamond site $x^\prime$, and let 
$Z_1 = (\text{col}_{x_0^\prime},t_{x_0},
\text{col}_{x_1^\prime},t_{x_1},\ldots )$ be a collection of arrival times 
at octagon sites and colours of diamond sites on the torus which induces
state
 $\ubx_0$ and such that
 $E_n$ holds. By definition, such a $Z_1$ exists. Let $Z_2$ be any other collection of octagon site arrival times and diamond site colours with state consistent with $\ubx_0$. 
At each even site of the torus, the first arrival time under $Z_2$ can be at most $\delta$ later than the first arrival at that site in $Z_1$, and similarly the first arrival at an odd site in $Z_2$ can be no more than $\delta$ earlier than the first arrival in $Z_1$. Therefore any sites which are black when all the arrival times at even sites in $Z_1$ are delayed by $2\delta$
(as per the definition of $E_n$) are also black in
 $Z_2$ (with no delay).
Since the existence of a horizontal crossing is black-increasing,
 and since $Z_1$ with a $2\delta$ delay on the arrival time at even sites
 has a horizontal black crossing of some $18n$ by $2n$ rectangle, $Z_2$
 must therefore have a horizontal black crossing of the same $18n$ by $2n$
 rectangle. Hence $F_n$ occurs and our assertion is justified.

  Suppose for some $\rho >0$ that
 $\limsup_{n \to \infty} h_{\rho} (n,\lambda,p) >0$.   
Let $\varepsilon_1 > 0$. 
By time rescaling 
$ \BP^{\mathbb{T}}_{\lambda(1+\eps),p+ \eps,1} (F_n)
= \BP^{\mathbb{T}}_{\lambda(1+\varepsilon)^{1/2},
p+ \eps,(1+\varepsilon)^{-1/2}} (F_n)$.
Hence by the event inclusion just proved,
 and Lemma  \ref{Lem0615},
we have infinitely often  (i.e., for infinitely many $n$) that
\begin{equation}\label{E3Ineq}
\BP^{\mathbb{T}}_{\lambda(1+\varepsilon),p+ \eps,1}
(F_n)=
\BP^{\mathbb{T}}_{\lambda(1+\varepsilon)^{1/2},p+ \eps,(1+\varepsilon)^{-1/2}}
(F_n)>1-\varepsilon_1.
\end{equation}

Now cover $\mathbb{T}$ with a set of $12n$ by $4n$ rectangles 
$R_{n,1},\ldots,R_{n,40}$ such that whenever $F_n$ holds there
 is a black path crossing some $R_{n,i}$ horizontally.
 We can do this by using rectangles with lower left corner having
 $x$-coordinate
 a multiple of $5n$ and $y$-coordinate a multiple of $2n$.
Let $H_{n,i}$ be the event that $R_{n,i}$ has a black horizontal crossing,
 and note that $H_{n,i}^c$ is white-increasing. Using the Harris-FKG inequality
(Lemma \ref{lemHarris}),  followed by
 Lemma \ref{TorApprox},
we have
 \begin{align*}
\BP^{\mathbb{T}}_{\lambda(1+\eps),p+\eps,1}
\left(\bigcap_{i=1}^{40} H_{n,i}^c\right)
& \geq
\prod_{i=1}^{40} \BP^{\mathbb{T}}_{\lambda(1+\eps),p+\eps,
1}(H_{n,i}^c)
\\
&= ( \BP^{\mathbb{T}}_{\lambda(1+\eps),p+\eps,1}
(H_{n,i}^c))^{40}
\\
&=(1-h_3(2n,\lambda(1+\eps),p+\eps)+o(1))^{40}.
\end{align*}
If none of the $H_{n,i}$ hold then $F_n$ fails, so by \eqref{E3Ineq},
infinitely often 
$$
1-h_3(2n,\lambda (1+\eps),p+\eps)
\leq\varepsilon_1^{1/40}+o(1),
$$
and hence
we have \eqref{mainResLimSup}.
\end{proof}

\begin{proof}[Proof of Proposition \ref{critequal}] 
Let $\lambda >0$, $ p \in (0,1)$ and $\eps \in
(0,1-p)$.  Suppose for some
  $\rho >0$ that
 $\limsup_{n \to \infty} h_{\rho} (n,\lambda,p) >0$.   
Choose $\eps_2 > 0$ such that $\eps_2 < \lambda$
and $\lambda +1  < \eps_2^{-1}$,
and $\eps_2 < p $ and $p+ \eps < 1-\eps_2$. Let $c_1= c_1(\eps_2)$ be
as in Proposition \ref{cor:halfCorrection} and assume without loss
of generality that $c_1 \geq 1$. 
Then 
by the first inequality of (\ref{eq:LambdaComp}),
 for all $n$ we have 
$$
h_3(n,\lambda,p+\varepsilon) \geq
 h_3(n,\lambda+\varepsilon/(2c_1),p+\varepsilon/2);
$$
 hence by Lemma  \ref{highProbCrossing}, we have (\ref{0815a}).
We prove (\ref{0815b}) similarly, now using the second inequality of
(\ref{eq:LambdaComp}).
\end{proof}

{\em Acknowledgement.}
We thank the referees for carefully reading an earlier version
of this manuscript, and providing many helpful suggestions.

The first author was supported by an EPSRC studentship.

\appendix
\label{appendix}
\section{Proof of Lemma \ref{leminfl}}

 By a continuity argument,
it suffices to prove the result
for the case where all entries of $\bp$  are dyadic  rationals, i.e.
 to show that 
for any $n$, $f$, $a$  and any $\bp$ with all entries dyadic rationals
satisfying (\ref{0721a}) we have (\ref{0721b}). 

Choose such a $\bp$ and choose $m \in \N$ such that all entries of
$2^m \bp$ are integers. Let $\BY$ be the space 
$\{0,1\}^m$ with the uniform distribution. We identify the space 
$\BX := \{0,1,\ldots,k\}$ under measure $\BP_{\bp}$, 
with the space $\BY$,
as follows.  Define a function $\tau:\BY \to \BX$ as
follows: the first $2^m p_0$ elements of $\BY$ (under the upwards 
lexicographic ordering) are mapped to
$0 \in \BX$, the next $2^m p_1$ elements of $\BY$ are
mapped to $1 \in \BX$, and so on.

Using this identification, any function
$g:\BX \to \{0,1\}$ induces another function
 $\tilde{g}:  \BY \to  \{0,1\}$,
given by $\tilde{g} = g \circ \tau$.
Moreover, for $\ell \in [m]$ the influence of the $\ell$th coordinate 
of a uniform random element of $\BY$ on $\tilde{g}$ is
equal to $w_{\ell,\bp}(g)$, since switching the $\ell$th digit
of the binary expansion of $U$ amounts to switching
the $\ell$th component of the corresponding random element
of $\BY$.
 Writing $w(\tilde{g})$ for
the sum (over $\ell$) of these influences,
 we have by Lemma \ref{KeyBoundUnlim} that
\bea
w(\tilde{g})
\leq 3 k^2 \pmax(\bp) \log(4/\pmax(\bp))
\leq 3 k^2 q \log (4/q).  
\label{0723a}
\eea
We identify $\BY$ with the power set of $[m]$ in the natural way.
For $S \in \BY$ (i.e.  for $S \subset [m]$), we set
$$
u_S(A) = (-1)^{|S \cap A|}, ~~~ A \subset [m].
$$  
It is well known (and not hard to prove)
 that the functions $u_S, S \subset [m]$ form an
orthonormal basis of the $2^m$-dimensional vector
space of functions from $\BY$ to $\R$, endowed with the
 inner product $\langle \cdot,\cdot \rangle$ given
by
$$
\langle g,h \rangle = 2^{-m} \sum_{A \subset [m]} g(A) h(A).
$$

Given functions $h$ and $g$ from $\BY$  to $\R$, define
the convolution $h * g$ by
\bea
h * g ( S ) = 2^{-m} \sum_{A \subset [m]} h(A) g(S \triangle A) , 
& &
 S \subset [m],
\label{convodef}
\eea
where $\triangle$ denotes the symmetric difference.
Also define the Walsh-Fourier transform $\hat{h}$  of $h$ by
\bea
\hat{h}(S) = \langle h, u_S \rangle,
& &
S \subset [m]. 
\label{WFdef}
\eea
Associated with this is the Walsh-Fourier expansion of $h$,
namely $h= \sum_S \hat{h}(S) u_S$, and the Parseval
equation $\|h\|_2^2:= \langle h,h \rangle = \sum_S \hat{h}(S)^2$.
These are both immediate from the fact that the $u_S$
form an orthonormal basis.
It is well known (and not hard to prove) that
for $S \subset [m]$ we have
\bea
\widehat{h * g} (S) =
 \hat{h}(S) \hat{g}(S). 
\label{convhat}
\eea

Define $T:\BY \to \R$ by $T(Z) = \sum_S u_S(Z) |S|^{1/2}$,
for $Z \subset [m]$, where the sum is over all  $S \subset [m]$.
Then $\hat{T}(S) = |S|^{1/2}$ for all $S$. 
Hence by (\ref{convhat}), for any $h: \BY \to \R$ we have
 $\widehat{T * h} (S) = \hat{h}(S)|S|^{1/2}$.
Hence by the Parseval identity,
\bea
\| T * h \|_2^2 = \sum_{S \subset [m]} \hat{h}(S)^2|S| = (1/4) w(h),
\label{Tgw}
\eea
where $w(h)$ is as in (\ref{0723a}) and for the last equality
we have used the first paragraph of \cite[p.73]{KKL}.

Fix $n \in \N$ and let $f: \BX^n \to \{0,1\}$ be a function.
Let $i \in [n]$. For $S_1,\ldots,S_{i-1},S_{i+1},\ldots,S_n \in \BY$,
 define the function 
$h=h [S_1,\ldots,S_{i-1},S_{i+1},\ldots,S_n]: \BY \to \{0,1\}$
by
\bea
h [S_1,\ldots,S_{i-1},S_{i+1},\ldots,S_n](S) = \tilde{f}(S_1,\ldots,S_{i-1},S,
S_{i+1},\ldots,S_n),
\label{hdef}
\eea
where we set $\tilde{f}(S_1,\ldots,S_n) := f(\tau(S_1),\ldots,\tau(S_n))$.
Also, define the function \linebreak $v=
v [S_1,\ldots,S_{i-1},S_{i+1},\ldots,S_n]: \BY \to \R$ by
$$
v [S_1,\ldots,S_{i-1},S_{i+1},\ldots,S_n] = T * 
h [S_1,\ldots,S_{i-1},S_{i+1},\ldots,S_n].
$$

Now define $W_i(S_1,\ldots,S_n) := v[S_1,\ldots,S_{i-1},S_{i+1},\ldots,S_n] 
(S_i)$, for $S_1,\ldots,S_n \subset [m]$ (recall that we are identifying
$\BY$ with the power set of $[m]$). Then
\bean
W_i(S_1,\ldots,S_n) & = & 2^{-m}
\sum_{R \subset [m]}
T(R) \tilde{f}(S_1,\ldots,S_{i-1}, S_i \triangle R,S_{i+1},\ldots,S_n) 
\\
& = & 2^{-m} \sum_{\tilde{R} \subset [mn] } T_i(\tilde{R}) 
\tilde{f}((S_1,\ldots,S_n) \triangle \tilde{R} )  
\eean
where for $R_1,\ldots,R_n \subset [m]$ we set
$T_i(R_1,\ldots,R_n)= T(R_i) $ if $R_j=\emptyset$ for all $j \neq i$
and $T_i(R_1,\ldots,R_n)=0$ otherwise. 
Thus, with convolutions of functions on $\BY^n$ (or equivalently,
on the power set of $[nm]$) defined analogously to (\ref{convodef}),
we have
\bea
W_i = 2^{m(n-1)} T_i * \tilde{f}.
\label{Wiconv}
\eea

For $F$ a real-valued function
on $\BY^n$ (or equivalently, on the power set of $[mn]$),
 we define the Walsh-Fourier transform of $F$ 
analogously to (\ref{WFdef}),
  by $\hat{F}(S) = 2^{-nm} \sum_{B \subset [nm]} u_S(B) F(B)$
for $S \subset [mn]$.
Writing $S=(S_1,\ldots,S_n)$ with $S_1,\ldots,S_n \subset [m]$,
and $B=(B_1,\ldots,B_n)$ similarly, we have
$u_S(B)= \prod_{j=1}^n u_{S_j}(B_j)$. Hence
\bean
\hat{T}_i(S_1,\ldots,S_n) 
& = & 
2^{-mn}  \sum_{B = (B_1,\ldots,B_n) \subset [mn] }
T_i(B) u_{S_1} (B_1) \cdots u_{S_n}(B_n) 
\\
& = & 2^{-mn} 
\sum_{B_i \subset [m] } T(B_i) u_{S_i}(B_i) 
\\
& = & 2^{-mn +m} 
\hat{T}( S_i) = 2^{m(1-n)} |S_i|^{1/2}.
\eean
Thus by (\ref{convhat}) and (\ref{Wiconv}),
 $\hat{W}_i(S_1,\ldots,S_n) = |S_i|^{1/2} \hat{\tilde{f}}(S_1,\ldots,S_n)$,
so by Parseval's equation for functions on $\BY^n$,
\bea
\| W_i\|_2^2 = \sum_{S_1,\ldots,S_n \subset [m]} 
(\hat{W}_i(S_1,\ldots,S_n))^2
\nonumber \\
=   \sum_{S_1,\ldots,S_n \subset [m]} |S_i| \hat{\tilde{f}}(S_1,\ldots,S_n)^2.
\label{0722a}
\eea
But also,
\bean
\|W_i\|_2^2 & = & 2^{-mn} \sum_{S_1,\ldots,S_n} 
( v[S_1,\ldots,S_{i-1},S_{i+1},\ldots, S_n] (S_i)) ^2
\\
& = & 2^{-mn} 2^m 
\sum_{S_1,\ldots,S_{i-1},S_{i+1},\ldots,S_n}  \| v[S_1,\ldots,S_{i-1},
S_{i+1}, \ldots, S_n] \|_2^2
\\
& = & 2^{m(1-n)} \sum_{S_1,\ldots,S_{i-1},S_{i+1},\ldots,S_n}
 w(h[S_1,\ldots,S_{i-1},S_{i+1},\ldots,S_n] )/4, 
\eean
where for the last line we have used (\ref{Tgw}).
By (\ref{0723a}),
$$
w(h[S_1,\ldots,S_{i-1},S_{i+1},\ldots,S_n]  ) \leq 
 3 k^2 q \log(4/q),
$$
and also
$w(h[S_1,\ldots,S_{i-1},S_{i+1},\ldots,S_n]  ) = 0$
if $\tilde{f}(S_1,\ldots,S_{i-1},\cdot,S_{i+1},\ldots,S_n)$
is  a constant function.  Hence,
\bean
\|W_i\|_2^2 \leq  
 (3/4)k^2 
q \log(4/q) 
I_{f,\bp}(i).
\eean
Summing over $i$ and
combining with (\ref{0722a}), 
we obtain that 
\bean
\sum_{\bS =(S_1,\ldots,S_n)}  \hat{\tilde{f}}(\bS)^2 
\|\bS\| \leq 
 (3/4)k^2 
q \log(4/q) 
\sum_{i=1}^n \delta_i, 
\eean
where we set
$\delta_i:= I_{f,\bp}(i)$ and $\|\bS \| := \sum_{i=1}^n |S_i|$.

Let $\cS_1 := \{\bS:  \|\bS\| \geq 2  k^2 (t(1-t))^{-1} 
q \log (4/q)
\sum_{i=1}^n \delta_i \}$. Then
\bea
\sum_{\bS \in \cS_1} \hat{\tilde{f}} (\bS)^2 & \leq &
 \frac{  t(1-t) \sum_{\bS \in \cS_1}
\|\bS\| \hat{\tilde{f}}(\bS)^2 }{ 
 2 k^2   
q \log (4/q)
 \sum_{i=1}^n \delta_i }
\nonumber \\
& \leq & \frac{3 t(1-t)}{8} ,
\label{0723d}
\eea
whereas by Parseval's equation,
since $\hat{\tilde{f}}(\emptyset) = \BE f (X) =h  $ and $f(\cdot) \in \{0,1\}$,
\bea
\sum_{\{\bS:\|\bS\| >0 \} }  \hat{\tilde{f}} (\bS)^2 =
 \| \tilde{f} \|_2^2- (\BE f (X) )^2 = 
t(1-t).
\label{0723e}
\eea

Next, for $i \in [n]$ we define the function $R_i$ on $\BY^n$ by
\bean
R_i & := & \sum_{S_1,\ldots,S_n \subset [m]:S_i \neq \emptyset} 
\hat{\tilde{f}} (S_1,\ldots,S_n) u_{S_1,\ldots,S_n}
\\ 
& = & \tilde{f} - 
 \sum_{S_1,\ldots,S_{i-1},S_{i+1},\ldots,S_n \subset [m]} 
\hat{\tilde{f}} (S_1,\ldots,\emptyset,\ldots ,S_n) u_{S_1,\ldots,
\emptyset,\ldots,S_n}
\eean   
where we have used the Walsh-Fourier expansion of $\tilde{f}$,
and where it is to be understood that the $\emptyset$ takes
the place of
$S_i$ in the sequence $(S_1,\ldots,\emptyset, \ldots,S_n)$. 
Now,
$$
\hat{\tilde{f}} (S_1,\ldots,\emptyset,\ldots,S_n) 
= \langle \tilde{f}, u_{(S_1\,\ldots,\emptyset, \ldots,S_n)} \rangle 
$$
\bean
& = & 2^{-mn} 
\sum_{B_1,\ldots,B_n \subset [m]}
\tilde{f}(B_1\ldots,B_n) \prod_{j:j \neq i} u_{S_j}(B_j)
\\
& = & 2^{m(1-n)}  \sum_{B_1,\ldots,B_{i-1},B_{i+1},\ldots,B_n \subset [m]}
g_i(B_1,\ldots,B_{i-1},B_{i+1},\ldots,B_n)
 \prod_{j:j \neq i} u_{S_j}(B_j), 
\eean
where we set $g_i(B_1,\ldots,B_{i-1},B_{i+1},\ldots,B_n)$
to be the value of $\tilde{f}(B_1,\ldots,B_n)$ averaged over all
values of $B_i$. Hence
\bean
\hat{\tilde{f}} (S_1,\ldots,\emptyset,\ldots,S_n) 
= \hat{g}_i (S_1,\ldots,S_{i-1},S_{i+1},\ldots,S_n)  
\eean
and so by a further Walsh-Fourier expansion, for
any $B_1,\ldots,B_n \subset [m]$ we have
\bean
R_i (B_1,\ldots,B_n) = \tilde{f}(B_1,\ldots,B_n) - 
g_i ( B_1,\ldots,B_{i-1},B_{i+1},\ldots,B_n).
\eean
Therefore $|R_i(\bB) | \leq 1$ for all $\bB = (B_1,\ldots,B_n) \subset [mn]$,
and $R_i(\bB) =0$ whenever $h[B_1,\ldots,B_{i-1},B_{i+1},\ldots,B_n]$,
defined by (\ref{hdef}),
is a constant function. Therefore, writing $\|g\|_p $ for
$(2^{-mn} \sum_{\bB \subset [mn]} |g(\bB)|^p)^{1/p}$
 for any real-valued function
$g$ defined on $\BY^n$ and any $p \geq 1$, we have that
$$
\| R_i \|_{4/3}^{4/3} \leq I_{f,\bp}(i),
$$
and therefore by the Bonami-Beckner inequality (Lemma 4 of \cite{BKKKL}),
 for $\eps = 3^{-1/2} $,
\bea
\|T_\eps R_i \|_2^2 \leq \|R_i\|_{1+\eps^2}^2 \leq \delta_i^{3/2},    
\label{0723b}
\eea
where we set
$$
T_\eps R_i := \sum_{{\bf S} \subset [mn]} \hat{R_i} ({\bf S}) \eps^{|{\bf S}|} 
u_{{\bf S}}.  
$$
Since $\hat{R}_i(S_1,\ldots,S_n)$ is zero or $\hat{\tilde{f}}(S_1,\ldots,S_n)$,
according to whether $S_i$ is empty or not,
so by Parseval's identity
\bea
\| T_\eps R_i \|_2^2 = \sum_{\bS = (S_1,\ldots,S_n)}
\hat{\tilde{f}}(\bS)^2 \eps^{2 \|\bS\| } {\bf 1}\{
S_i \neq \emptyset\} .
\label{0723c}
\eea
For $\bS =(S_1,\ldots,S_n) \subset [mn]$ let
 $\mu (\bS)$ denote the number of $i$ such that $S_i \neq \emptyset$.
Comparing (\ref{0723b}) with (\ref{0723c}) and
summing over $i$ yields
$$
 \sum_{\bS }
\hat{\tilde{f}}(\bS)^2 \eps^{2 \| \bS\| } \mu(\bS)
\leq \sum_{i=1}^n \delta_i^{3/2}.
$$
Let ${\cS}_2$ be the set of $\bS$ such that
$1 > \eps^{2\|\bS\|} \geq \left( 2 \sum_i \delta_i^{3/2} \right)/ (t(1-t))$.  
Then 
\bean
\sum_{\bS \in \cS_2} \hat{\tilde{f}}(\bS)^2 \leq
\sum_{\bS \in \cS_2} \frac{ \mu(\bS) \eps^{2\|\bS\| } \hat{\tilde{f}}(\bS)^2
t(1-t) }{
 2 \sum_i \delta_i^{3/2} }
\\
\leq t(1-t)/2.
\eean
Combined with (\ref{0723d}) and (\ref{0723e}),
since $(3/8)+(1/2)< 1$, this shows that
there exists $\bS$ with $\|\bS \| >0$ lying  neither in
$\cS_1$ nor in $\cS_2$. Choosing such an $\bS$, 
since $\bS \notin \cS_2$ 
we have
$
3^{-\|\bS\|} < \left( 2 \sum_i \delta_i^{3/2} \right)/(t(1-t))
$ 
so that
$$
\| \bS \| > \log \left( \frac{ t(1-t)}{2 \sum_i \delta_i^{3/2} }
 \right)/ \log 3   
$$
but also $\bS \notin \cS_1$, so that
$$
 2 k^2 
q \log (4/q)  
\sum_i \delta_i >  t(1-t) \log \left(
\frac{t(1-t)}{2 \sum \delta_i^{3/2} } \right)/ \log 3.
$$

Suppose (\ref{0721a}) holds. Then, setting 
$\alpha := a q^2 (\log(4/q))^2$ 
and 
$I_f := \sum_{i=1}^n \delta_i$, 
we have that
$\delta_i \leq \alpha$ for all $i$ so that
 $
\sum_i \delta_i^{3/2} \leq 
 \alpha^{1/2} I_f. 
$
Since $2 \log 3 <3$, we have that
$$
3 k^2 I_f q \log (4/q) 
>  t(1-t) \log \left( \frac{
t(1-t)}{2 \alpha^{1/2} I_f} \right).
$$
Setting $x := t(1-t)/(q \log (4/q))$ and $b := I_f/x$, we have
$$
I_f > \left( \frac{ x}{ 3 k^2} \right) \log \left(
\frac{t(1-t)}{2 a^{1/2} 
(q \log 4/q) 
b x} \right) 
= \left( \frac{ x}{3 k^2} \right) \log \left( \frac{1}{2 a^{1/2}b} \right).
$$
Since $I_f = bx$ it follows that 
$b \geq (1/3)k^{-2} \log(1/(2a^{1/2} b))$, and therefore
$b+ (1/3) k^{-2} \log b \geq  (1/3)k^{-2} \log (1/(2a^{1/2}))$.

Since $(\log u)/u \leq e^{-1}$ for all $u >0$, 
and since we assume
 $a \leq 1/16$ so that $\log (a^{-1/2} ) \geq 2 \log 2$,
therefore
$$
2 b \geq b + (1/3)k^{-2} \log b \geq (1/3)  k^{-2} \log (1/(2a^{1/2})) 
\geq 
 (1/6) k^{-2} \log a^{-1/2}. 
$$ 
Therefore
$b \geq
 (24 k^2)^{-1}
\log (1/a)$, which implies (\ref{0721b}). 

\end{document}